\documentclass[11pt]{amsart}                                                   
\usepackage{amsmath,amssymb,latexsym, geometry, amsthm, amsfonts, multicol, color}                                           
                                                                                
\definecolor{gold}{rgb}{0.85,.66,0}
\definecolor{cherry}{rgb}{0.9,.1,.2}
\definecolor{burgundy}{rgb}{0.8,.2,.2}
\definecolor{orangered}{rgb}{0.85,.3,0}
\definecolor{orange}{rgb}{0.85,.4,0}
\definecolor{olive}{rgb}{.45,.4,0}
\definecolor{lime}{rgb}{.6,.9,0}
\definecolor{green}{rgb}{.2,.7,0}
\definecolor{grey}{rgb}{.4,.4,.2}
\definecolor{brown}{rgb}{.4,.3,.1}

\newtheorem{theorem}{Theorem}[section]
\newtheorem{corollary}[theorem]{Corollary}
\newtheorem{lemma}[theorem]{Lemma}
\newtheorem{proposition}[theorem]{Proposition}

\newtheorem{definition}[theorem]{Definition}

\newcommand{\EM}{\sf{m}}

\title{A simple bijection between standard $3 \times n$ tableaux and irreducible webs for $\mathfrak{sl}_3$}
\author{Julianna Tymoczko}
 \address{Mathematics Department, University of Iowa, 14 MacLean
 Hall, Iowa City, Iowa 52242-1419, USA.}
 \email{julianna-tymoczko@uiowa.edu}
 \subjclass[2000]{Primary: 05E10, 05C10}
 \keywords{spider, representations of Lie algebras, Young tableau, jeu de taquin, promotion}
 \thanks{The author is partially supported by 
 NSF grant DMS-0801554 and a Sloan Research Fellowship.}

\begin{document}

\maketitle

\begin{abstract}
Combinatorial spiders are a model for the invariant space of the tensor product of representations.  The basic objects, {\em webs}, are certain directed planar graphs with boundary; algebraic operations on representations correspond to graph-theoretic operations on webs.  Kuperberg developed spiders for rank $2$ Lie algebras and $\mathfrak{sl}_2$.  Building on a result of Kuperberg's, Khovanov-Kuperberg found a recursive algorithm giving a bijection between standard Young tableaux of shape $3 \times n$ and irreducible webs for $\mathfrak{sl}_3$ whose boundary vertices are all sources.

In this paper, we give a simple and explicit map from standard Young tableaux of shape $3 \times n$ to  irreducible webs for $\mathfrak{sl}_3$ whose boundary vertices are all sources, and show that it is the same as Khovanov-Kuperberg's map.  Our construction generalizes to some webs with both sources and sinks on the boundary.  Moreover, it allows us to extend the correspondence between webs and tableaux in two ways.  First, we provide a short, geometric proof of Petersen-Pylyavskyy-Rhoades's recent result that rotation of webs corresponds to jeu-de-taquin promotion on $3 \times n$ tableaux.  Second, we define another natural operation on tableaux called a {\em shuffle}, and show that it corresponds to the join of two webs.  Our main tool is an intermediary object between tableaux and webs that we call an $\EM$-diagram.  The construction of $\EM$-diagrams, like many of our results, applies to shapes of tableaux other than $3 \times n$.
\end{abstract}

\section{Introduction}

Spiders are categories that describe representations of Lie algebras, particularly the invariant space of a tensor product of irreducible representations.  Kuperberg introduced a combinatorial description of spiders for all rank $2$ Lie algebras, as well as for $\mathfrak{sl}_2$, in which representations correspond to combinatorial graphs (called {\em webs}), and algebraic operations (like permutation of the tensor factors) correspond to combinatorial operations on the graphs \cite{Kup96}.  A web for the $\mathfrak{sl}_3$-spider is a planar directed graph embedded in a disk so that (1) internal vertices are trivalent and boundary vertices have degree one, and (2) each vertex is either a source (all edges directed out of the vertex) or a sink (all edges directed in).  (We use the streamlined presentation of Petersen-Pylyavskyy-Rhoades \cite{PPR09}.)  This construction seems like it could be generalized to other Lie algebras, yet combinatorial spiders are only known in the cases Kuperberg originally identified.  Researchers have recently and independently made suggestive inroads into this important open problem, including Kim \cite{Kim03} and Morrison \cite{Mor07}, and Jeong-Kim \cite{JK}.  

Young tableaux are a classical construction ubiquitous in the representation theory and geometry associated to the symmetric group $S_N$ and the Lie algebra $\mathfrak{sl}_N$ \cite[Part II]{Ful97}.  The Young diagram corresponding to the partition $\lambda_1 \leq \lambda_2 \leq \cdots$ of $N$ is a left-justified array with $\lambda_1$ boxes in the top row, $\lambda_2$ boxes in the second row, and so on.  A standard Young tableau corresponding to the partition $\lambda_1 \leq \lambda_2 \leq \cdots$ of $N$ is a filling of the Young diagram by the numbers $1,2,\ldots, N$ without repetition so that numbers increase left-to-right along rows and bottom-to-top along columns.  

Our paper deepens the connections between Young tableaux and spiders, placing these newer constructions in the context of classical work.   These connections have some precedent. Fung constructed a natural bijection between standard Young tableaux of shape $(n, n)$ and irreducible webs for $\mathfrak{sl}_2$ using the geometry of an object called the $(n,n)$ Springer variety \cite{Fun03}.  Khovanov-Kuperberg constructed a bijection between Young tableaux of shape $(n, n, n)$ and irreducible webs for $\mathfrak{sl}_3$ for which each boundary vertex is a source \cite{KK99}.  (Choosing whether a boundary vertex is a source or a sink is equivalent to choosing whether the corresponding tensor factor of the representation is the fundamental representation for $\mathfrak{sl}_3$ or its dual.)  However, Khovanov-Kuperberg's proof uses a complicated set of {\em growth rules} which, when recursively applied, eventually generate all irreducible webs for $\mathfrak{sl}_3$. 

This paper gives a simple and direct map from standard Young tableaux of shape $(n,n,n)$ to irreducible webs for $\mathfrak{sl}_3$ whose boundary vertices are all sources.  We give a quick example and colloquial description here; the reader interested in details can read Sections \ref{section: m definition}, \ref{section: resolved m definition}, and \ref{subsection: different depths} immediately, together with Sections \ref{section: m-diagram examples}, \ref{section: web examples}, and \ref{section: depth examples} for examples.

Our map uses an intermediate object called an $\EM$-diagram.  The $\EM$-diagram for a standard tableau with $N$ boxes consists of a boundary line with the numbers $1,2,\ldots,N$, together with a collection of arcs drawn above it.  To draw the arcs, read from the bottom to the top row, and then from left to right along each row, connecting the number $i$ with an arc to the (1) largest number (2) on the row below $i$ that (3) is not yet connected to a number on $i$'s row.  (Section \ref{section: Young to m} has more.) 
\vspace{-2em}
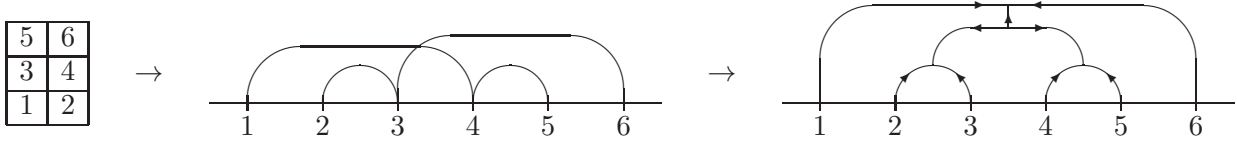
\[\begin{array}{|c|c|}
\cline{1-2}
5 & 6\\
\cline{1-2}
3 & 4\\
\cline{1-2}
1 & 2\\
\cline{1-2}
\end{array}
\hspace{0.2in} \rightarrow \hspace{0.2in}
\begin{picture}(60, 10)(0, 8)
\put(0,5){\line(1,0){60}}
\put(5,4){\line(0,1){2}}
\put(15,4){\line(0,1){2}}
\put(25,4){\line(0,1){2}}
\put(35,4){\line(0,1){2}}
\put(45,4){\line(0,1){2}}
\put(55,4){\line(0,1){2}}
\put(5,2){\makebox(0,0){$1$}}
\put(15,2){\makebox(0,0){$2$}}
\put(25,2){\makebox(0,0){$3$}}
\put(35,2){\makebox(0,0){$4$}}
\put(45,2){\makebox(0,0){$5$}}
\put(55,2){\makebox(0,0){$6$}}
\put(20,5){\oval(10,10)[t]}
\put(40,5){\oval(10,10)[t]}
\put(40,5){\oval(30,18)[t]}
\put(20,5){\oval(30,15)[t]}
\end{picture}
\hspace{0.2in} \rightarrow \hspace{0.2in}
\begin{picture}(60, 15)(0, 8)
\put(0,5){\line(1,0){60}}
\put(5,4){\line(0,1){2}}
\put(15,4){\line(0,1){2}}
\put(25,4){\line(0,1){2}}
\put(35,4){\line(0,1){2}}
\put(45,4){\line(0,1){2}}
\put(55,4){\line(0,1){2}}
\put(5,2){\makebox(0,0){$1$}}
\put(15,2){\makebox(0,0){$2$}}
\put(25,2){\makebox(0,0){$3$}}
\put(35,2){\makebox(0,0){$4$}}
\put(45,2){\makebox(0,0){$5$}}
\put(55,2){\makebox(0,0){$6$}}
\put(20,5){\oval(10,10)[t]}
\put(40,5){\oval(10,10)[t]}
\put(30,10){\oval(20,10)[t]}
\put(30,5){\oval(50,26)[t]}
\put(30,15){\line(0,1){3}}
\put(25,18){\vector(1,0){2}}
\put(35,18){\vector(-1,0){2}}
\put(30,15){\vector(0,1){2}}
\put(27,15){\vector(-1,0){2}}
\put(33,15){\vector(1,0){2}}
\put(15,7){\vector(1,1){2}}
\put(25,7){\vector(-1,1){2}}
\put(35,7){\vector(1,1){2}}
\put(45,7){\vector(-1,1){2}}
\end{picture}
\]

\noindent To construct a web from this $\EM$-diagram, do three things.  First, each boundary vertex on two arcs looks locally like a $V$; replace this neighborhood with a small $Y$.  Second, each arc now has exactly one endpoint on the boundary; direct each arc away from the boundary vertex, continuing with the same direction across any intersections.  Third, anywhere two arcs cross is a four-valent vertex with two edges directed in and two directed out; replace this vertex with two vertices joined by a directed edge.  We will confirm that there is a unique way to do this so that one of the new vertices is a source and the other is a sink.  (Section \ref{section: m to web} has more detail and precise definitions.)

A series of lemmas in Section \ref{section: reduced webs} prove that the planar graphs obtained in this way from standard Young tableaux of shape $(n \leq k \leq k)$ are in fact irreducible webs for $\mathfrak{sl}_3$.  (This partially generalizes Khovanov-Kuperberg's work.)  Moreover, we prove in Theorem \ref{theorem: bijection} that this map is a bijection between standard Young tableaux of shape $(n,n,n)$ and irreducible webs for $\mathfrak{sl}_3$ whose boundary vertices are all sources. In fact, we will show that our bijection coincides with Khovanov-Kuperberg's bijection; we also show that it can be extended to some irreducible webs for $\mathfrak{sl}_3$ with both sources and sinks as boundary vertices.  

Theorem \ref{theorem: bijection} actually proves that Khovanov-Kuperberg's map from webs to tableaux inverts our map from tableaux to webs. Our proof uses two notions of depth: {\em circle depth} in an $\EM$-diagram, which is the number of arcs above each face of the $\EM$-diagram; and Khovanov-Kuperberg's {\em path depth} in a web, which is the minimum number of edges crossed by paths from a given face to the unbounded face of the web. (Path depth is distance in the dual graph to a planar graph.)  
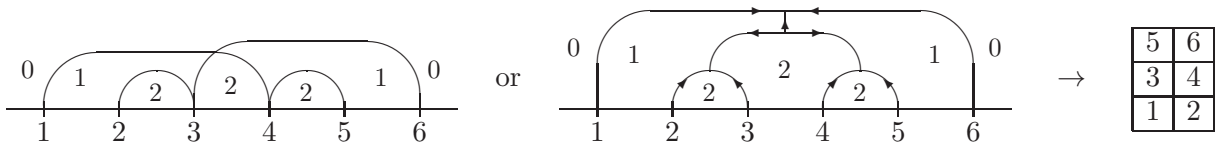
\vspace{-1em}
\[
\begin{picture}(60, 10)(0, 8)
\put(0,5){\line(1,0){60}}
\put(5,4){\line(0,1){2}}
\put(15,4){\line(0,1){2}}
\put(25,4){\line(0,1){2}}
\put(35,4){\line(0,1){2}}
\put(45,4){\line(0,1){2}}
\put(55,4){\line(0,1){2}}
\put(5,2){\makebox(0,0){$1$}}
\put(15,2){\makebox(0,0){$2$}}
\put(25,2){\makebox(0,0){$3$}}
\put(35,2){\makebox(0,0){$4$}}
\put(45,2){\makebox(0,0){$5$}}
\put(55,2){\makebox(0,0){$6$}}
\put(20,5){\oval(10,10)[t]}
\put(40,5){\oval(10,10)[t]}
\put(40,5){\oval(30,18)[t]}
\put(20,5){\oval(30,15)[t]}
\put(19,6){\small 2}
\put(29,7){\small 2}
\put(39,6){\small 2}
\put(9,8){\small 1}
\put(49,8){\small 1}
\put(2,9){\small 0}
\put(56,9){\small 0}
\end{picture}
\hspace{0.2in} \textup{or} \hspace{0.2in}
\begin{picture}(60, 15)(0, 8)
\put(0,5){\line(1,0){60}}
\put(5,4){\line(0,1){2}}
\put(15,4){\line(0,1){2}}
\put(25,4){\line(0,1){2}}
\put(35,4){\line(0,1){2}}
\put(45,4){\line(0,1){2}}
\put(55,4){\line(0,1){2}}
\put(5,2){\makebox(0,0){$1$}}
\put(15,2){\makebox(0,0){$2$}}
\put(25,2){\makebox(0,0){$3$}}
\put(35,2){\makebox(0,0){$4$}}
\put(45,2){\makebox(0,0){$5$}}
\put(55,2){\makebox(0,0){$6$}}
\put(20,5){\oval(10,10)[t]}
\put(40,5){\oval(10,10)[t]}
\put(30,10){\oval(20,10)[t]}
\put(30,5){\oval(50,26)[t]}
\put(30,15){\line(0,1){3}}
\put(25,18){\vector(1,0){2}}
\put(35,18){\vector(-1,0){2}}
\put(30,15){\vector(0,1){2}}
\put(27,15){\vector(-1,0){2}}
\put(33,15){\vector(1,0){2}}
\put(15,7){\vector(1,1){2}}
\put(25,7){\vector(-1,1){2}}
\put(35,7){\vector(1,1){2}}
\put(45,7){\vector(-1,1){2}}
\put(19,6){\small 2}
\put(29,9){\small 2}
\put(39,6){\small 2}
\put(9,11){\small 1}
\put(49,11){\small 1}
\put(1,12){\small 0}
\put(57,12){\small 0}
\end{picture}
\hspace{0.2in} \rightarrow \hspace{0.2in}
\begin{array}{|c|c|}
\cline{1-2}
5 & 6\\
\cline{1-2}
3 & 4\\
\cline{1-2}
1 & 2\\
\cline{1-2}
\end{array}
\]
In Lemma \ref{lemma: circle depth is path depth}, we prove that these two depths coincide in an appropriate sense.  Khovanov-Kuperberg's map from webs to tableaux puts $i$ on the bottom row of the tableau if the depth increases at the boundary vertex $i$, the middle row if depth stays the same at $i$, and the top row if depth decreases at $i$.  The reader can see in our example that this recovers the original tableau.

We provide two applications of our construction. Proposition \ref{proposition: promotion and rotation} radically simplifies Petersen-Pylyavskyy-Rhoades's recent result for $(n,n,n)$ tableaux that an operation on tableaux called {\em jeu-de-taquin promotion} corresponds to an operation on webs called {\em rotation} \cite{PPR09}.  Proposition \ref{proposition: shuffle and join} gives the new result that an operation called the {\em join} of two webs corresponds to an operation on Young tableaux called a {\em shuffle}.  (The proposition applies to more general shapes than three-row Young tableaux, though the corresponding planar graphs are not webs in that case.)

Though not directly relevant to this work, our construction is motivated by the geometry of the $(n,n,n)$ Springer variety, and generalizes earlier work of Fung's \cite{Fun03} and of the author's with H.~Russell \cite{RT}.  We ask two questions for future research: 
\begin{enumerate}
\item Does this method give webs for $\mathfrak{sl}_k$ from standard Young tableaux of shape $k \times n$?
\item What does this correspondance (particularly depth) imply about the geometry of the irreducible components of the $(n,n,n)$ Springer variety? 
\end{enumerate}

\section{From Young tableaux to $\EM$-diagrams}\label{section: Young to m}
Our path from Young tableaux to webs goes through an object which we call an {\em $\EM$-diagram}.  The $\EM$-diagrams are read directly from the Young tableaux; they are almost webs, except that they have 4-valent vertices that we will turn into trivalent vertices.  In this section, we describe how to construct $\EM$-diagrams, give several examples, and then prove fundamental properties of $\EM$-diagrams.
\subsection{Defining $\EM$-diagrams}  \label{section: m definition}

Let $\lambda$ be a Young diagram with $N$ boxes and let $T$ be a standard tableau of shape $\lambda$.
Construct the $\EM$-diagram corresponding to $T$ as follows: 
\begin{enumerate}
\item Draw a line with the numbers $1,2,\ldots, N$ in increasing order.  This is the boundary line at the base of the $\EM$-diagram; all arcs are drawn above this line.
\item For each $i=1,2,\ldots,N$ not on the bottom row, find $j<i$ such that $j$ is
\begin{itemize}
\item the largest number
\item that lies on the row immediately below the row with $i$ and
\item that is not already on an arc with another number from the same row as $i$.
\end{itemize}
In other words, 
\[j = \max \{k \textup{ on the row below $i$}: k < i, \textup{ $k$ not on an arc to a number on the same row as $i$}\}.\] 
Then 
\begin{itemize}
\item join $i$ to $j$ with a semicircular arc.
\end{itemize}
\end{enumerate}

For instance, if $i$ is on the bottom row of $T$ then there is no such $j$, and no arc is created at the $i^{th}$ iteration of Step (2).  For readers who prefer visual descriptions, the number $j$ is the first number to the left of $i$ on the boundary line such that $j$ is on the row below $i$ and $j$ is not joined by an arc to any number on the same row as $i$.  

\begin{definition}
Arcs between the $k^{th}$ row from the bottom and $(k+1)^{st}$ row from the bottom are called {\em $k^{th}$ arcs}.  A {\em boundary vertex} is a point lying on both an arc and the boundary line.
\end{definition}

We use three-row tableaux extensively.  The following terminology is useful in this special case.

\begin{definition}
Suppose $i < j<k$ are boundary vertices.  If the only arc incident to either $i$ or $j$ is $(i,j)$, then $(i,j)$ is called an {\em isolated arc}.  If the only arcs incident to any of $i$, $j$, or $k$ are $(i,j)$ and $(j,k)$ then $(i,j,k)$ is called an {\em $\EM$}.
\end{definition}

\subsection{Examples} \label{section: m-diagram examples}

This paper focuses on Young diagrams with three rows, usually rectangular.  An $\EM$-diagram for a $3 \times n$ Young tableau contains $n$ figures, each of which resembles an $\EM$ (and is called an $\EM$).  The $\EM$s can be nested, unnested, or cross in various ways.  In this section we give the $\EM$-diagrams for each of the five possible $3 \times 2$ tableaux, with second arcs drawn in boldface. (Corollary \ref{corollary:two ms relative} will prove that {\em any} two $\EM$s in {\em any} $\EM$-diagram for a three-row tableau are in one of these relative positions.)

The simplest kind of web corresponds to the tableau filled with $\{1,2,\ldots\}$ in numerical order, bottom to top and left to right.  The arcs in these webs neither cross nor nest each other.
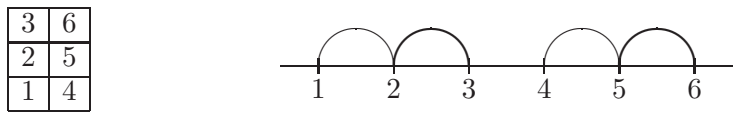
\[
\begin{array}{|c|c|}
\cline{1-2}
3 & 6\\
\cline{1-2}
2 & 5\\
\cline{1-2}
1 & 4\\
\cline{1-2}
\end{array} 
\hspace{1in} 
\begin{picture}(60, 10)(0, 5)
\put(0,5){\line(1,0){60}}
\put(5,4){\line(0,1){2}}
\put(15,4){\line(0,1){2}}
\put(25,4){\line(0,1){2}}
\put(35,4){\line(0,1){2}}
\put(45,4){\line(0,1){2}}
\put(55,4){\line(0,1){2}}
\put(5,2){\makebox(0,0){$1$}}
\put(15,2){\makebox(0,0){$2$}}
\put(25,2){\makebox(0,0){$3$}}
\put(35,2){\makebox(0,0){$4$}}
\put(45,2){\makebox(0,0){$5$}}
\put(55,2){\makebox(0,0){$6$}}

\thinlines
\put(10,5){\oval(10,10)[t]}
\put(40,5){\oval(10,10)[t]}

\thicklines
\put(20,5){\oval(10,10)[t]}
\put(50,5){\oval(10,10)[t]}
\end{picture}
\]
The next two examples demonstrate nesting: no arcs cross, but one $\EM$ sits inside another.  This can happen in more than one way: either a first or second arc may enclose an $\EM$.
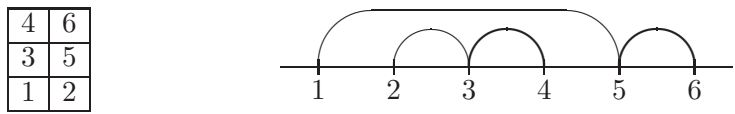
\[\begin{array}{|c|c|}
\cline{1-2}
4 & 6\\
\cline{1-2}
3 & 5\\
\cline{1-2}
1 & 2\\
\cline{1-2}
\end{array}
\hspace{1in} 
\begin{picture}(60, 10)(0, 5)
\put(0,5){\line(1,0){60}}
\put(5,4){\line(0,1){2}}
\put(15,4){\line(0,1){2}}
\put(25,4){\line(0,1){2}}
\put(35,4){\line(0,1){2}}
\put(45,4){\line(0,1){2}}
\put(55,4){\line(0,1){2}}
\put(5,2){\makebox(0,0){$1$}}
\put(15,2){\makebox(0,0){$2$}}
\put(25,2){\makebox(0,0){$3$}}
\put(35,2){\makebox(0,0){$4$}}
\put(45,2){\makebox(0,0){$5$}}
\put(55,2){\makebox(0,0){$6$}}

\thinlines
\put(20,5){\oval(10,10)[t]}
\put(25,5){\oval(40,15)[t]}

\thicklines
\put(30,5){\oval(10,10)[t]}
\put(50,5){\oval(10,10)[t]}
\end{picture}
\]
\[\begin{array}{|c|c|}
\cline{1-2}
5 & 6\\
\cline{1-2}
2 & 4\\
\cline{1-2}
1 & 3\\
\cline{1-2}
\end{array}
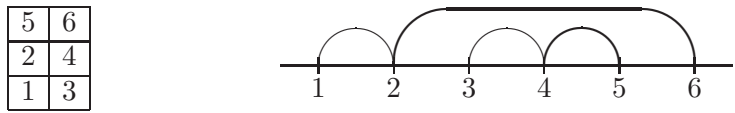
\hspace{1in} 
\begin{picture}(60, 10)(0, 5)
\put(0,5){\line(1,0){60}}
\put(5,4){\line(0,1){2}}
\put(15,4){\line(0,1){2}}
\put(25,4){\line(0,1){2}}
\put(35,4){\line(0,1){2}}
\put(45,4){\line(0,1){2}}
\put(55,4){\line(0,1){2}}
\put(5,2){\makebox(0,0){$1$}}
\put(15,2){\makebox(0,0){$2$}}
\put(25,2){\makebox(0,0){$3$}}
\put(35,2){\makebox(0,0){$4$}}
\put(45,2){\makebox(0,0){$5$}}
\put(55,2){\makebox(0,0){$6$}}

\thinlines
\put(30,5){\oval(10,10)[t]}
\put(10,5){\oval(10,10)[t]}

\thicklines
\put(40,5){\oval(10,10)[t]}
\put(35,5){\oval(40,15)[t]}
\end{picture}
\]
Finally, two arcs may cross, as in the last two examples.  This can also happen in two different ways: either the second arc crosses from above the first arc, or from below.
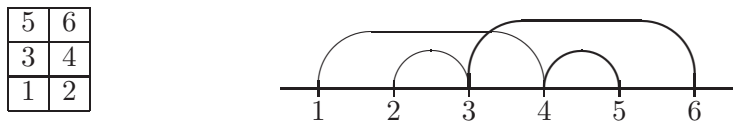
\[\begin{array}{|c|c|}
\cline{1-2}
5 & 6\\
\cline{1-2}
3 & 4\\
\cline{1-2}
1 & 2\\
\cline{1-2}
\end{array}
\hspace{1in} 
\begin{picture}(60, 10)(0, 8)
\put(0,5){\line(1,0){60}}
\put(5,4){\line(0,1){2}}
\put(15,4){\line(0,1){2}}
\put(25,4){\line(0,1){2}}
\put(35,4){\line(0,1){2}}
\put(45,4){\line(0,1){2}}
\put(55,4){\line(0,1){2}}
\put(5,2){\makebox(0,0){$1$}}
\put(15,2){\makebox(0,0){$2$}}
\put(25,2){\makebox(0,0){$3$}}
\put(35,2){\makebox(0,0){$4$}}
\put(45,2){\makebox(0,0){$5$}}
\put(55,2){\makebox(0,0){$6$}}

\thinlines
\put(20,5){\oval(30,15)[t]}
\put(20,5){\oval(10,10)[t]}

\thicklines
\put(40,5){\oval(10,10)[t]}
\put(40,5){\oval(30,18)[t]}
\end{picture}
\]
\[\begin{array}{|c|c|}
\cline{1-2}
4 & 6\\
\cline{1-2}
2 & 5\\
\cline{1-2}
1 & 3\\
\cline{1-2}
\end{array}
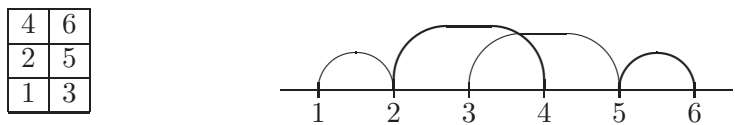
\hspace{1in} 
\begin{picture}(60, 10)(0, 8)
\put(0,5){\line(1,0){60}}
\put(5,4){\line(0,1){2}}
\put(15,4){\line(0,1){2}}
\put(25,4){\line(0,1){2}}
\put(35,4){\line(0,1){2}}
\put(45,4){\line(0,1){2}}
\put(55,4){\line(0,1){2}}
\put(5,2){\makebox(0,0){$1$}}
\put(15,2){\makebox(0,0){$2$}}
\put(25,2){\makebox(0,0){$3$}}
\put(35,2){\makebox(0,0){$4$}}
\put(45,2){\makebox(0,0){$5$}}
\put(55,2){\makebox(0,0){$6$}}

\thinlines
\put(10,5){\oval(10,10)[t]}
\put(35,5){\oval(20,15)[t]}

\thicklines
\put(50,5){\oval(10,10)[t]}
\put(25,5){\oval(20,17)[t]}
\end{picture}
\]

With larger tableaux, we will often see combinations of crossing and nesting within one $\EM$-diagram.  However, any two $\EM$s will be in the same relative position as one of these five examples. 

\subsection{Properties of $\EM$s}
We now prove several key properties about how $\EM$s can cross in an $\EM$-diagram.  First we confirm that $\EM$-diagrams are well-defined.

\begin{lemma}
The map from the standard tableau $T$ of shape $\lambda$ to an $\EM$-diagram is well-defined.
\end{lemma}

\begin{proof}
Suppose that $i$ is the $k^{th}$ box in its row.  Columns increase in a standard Young tableau, so $i$ is greater than the $k^{th}$ box in the row below $i$.  Rows increase, so $i$ is greater than all of the first $k$ boxes in the row below $i$.  By construction, each number in the row containing $i$ is joined to at most one number in the row below, so there is at least one number $j < i$ in the row below $i$ and to the left of $i$ which is not part of an arc.
\end{proof}

\begin{proposition}\label{proposition:basic noncrossing props}
The $\EM$-diagram of a standard Young tableau of shape $\lambda$ satisfies the following:
\begin{enumerate}
\item Two arcs intersect in at most one point.   In particular, locally the $\EM$-diagram is an X near each crossing; the arcs share no point other than the intersection point.
\item The set of $k^{th}$ arcs is pairwise noncrossing.
\end{enumerate}
\end{proposition}

\begin{proof}
Each arc is a semicircle whose center and endpoints are on the boundary line, and whose diameter is the distance between the endpoints.  

Two distinct circles intersect in at most two points.  If we treat the boundary line as the x-axis, then the two points of intersection have coordinates $(x_0, y_0)$ and $(x_0, -y_0)$.  At most one of these points lies above the boundary line, namely on the arc.  So any two arcs intersect in at most one point, and locally near the intersection, the $\EM$-diagram is an X.  

Suppose that $(i,j)$ and $(i',j')$ are the endpoints of two intersecting arcs, with $i<j$ and $i'<j'$.  Without loss of generality assume that $i<i'\leq j < j'$.  If $i'=j$ then $(i,j)$ is a $k^{th}$ arc and $(i',j')$ is a $(k+1)^{th}$ arc for some $k$.  If not, then $j$ and $j'$ are not both on the $(k+1)^{th}$ row, since by construction of $\EM$-diagrams $j$ must be joined to the largest possible $i'$ from the $k^{th}$ row.  This proves the second part of the claim.  
\end{proof}

These conditions seem similar to those that arise in recent work of Petersen-Pylyavskyy-Speyer \cite{PPS}.  The next corollary specializes to Young diagrams with three rows.

\begin{corollary}\label{corollary:two ms relative}
The $\EM$-diagram of a standard three-row Young tableau satisfies the following:
\begin{enumerate}
\item At most two arcs intersect at a given point.
\item Two arcs that cross consist of a first arc (of an $\EM$ or an isolated arc) and the second arc of a different $\EM$.
\item Any two $\EM$s cross at most once. 
\item Any two $\EM$s will be in one of the five relative positions described by the $\EM$-diagrams of $3 \times 2$ tableaux.
\end{enumerate}
\end{corollary}

\begin{proof}

The only kinds of arcs in $\EM$-diagrams from three-row Young tableaux are first arcs (including isolated arcs) and second arcs. By Proposition \ref{proposition:basic noncrossing props}, if two arcs cross, then they must be a first arc (of an $\EM$ or an isolated arc) and the second arc of another $\EM$.  In particular, no more than two arcs intersect in the same point, since each arc is either a first arc or a second arc. This proves the first two parts of the theorem.

Suppose two arcs cross as sketched below.  By above, one arc is the first arc of an $\EM$, and the other arc is the second arc of another $\EM$.  
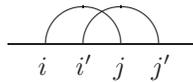
\[\begin{picture}(25,15)(0,0)
\put(0,5){\line(1,0){25}}
\put(10,5){\oval(10,10)[t]}
\put(15,5){\oval(10,10)[t]}
\put(4,1){\small $i$}
\put(9,1){\small $i'$}
\put(14,1){\small $j$}
\put(19,1){\small $j'$}
\end{picture}\]

Suppose $(i,j)$ is the second arc of the $\EM$ with vertices $(k,i,j)$ and $(i',j')$ is the first arc of the $\EM$ with vertices $(i',j',k')$.  Then $k<i<i'$ and $j<j'<k'$, so these $\EM$s cross only once.  

Suppose $(i,j)$ is the first arc of the $\EM$ with boundary points $(i,j,k)$ and $(i',j')$ is the second arc of the $\EM$ with boundary points $(k',i',j')$.  Since second arcs cannot intersect, the endpoints satisfy $k<j'$.  Similarly, since first arcs do not cross, the initial points satisfy $i<k'$.  Hence these $\EM$s cross only once.

In no case can two $\EM$s cross twice, proving the next part of the claim.  Moreover, the $\EM$-diagrams that we obtained are precisely the two crossing $\EM$-diagrams from $3 \times 2$ tableaux.

Suppose two $\EM$s $(i,j,k)$ and $(i',j',k')$ do not cross.  Assume without loss of generality that $i<i'$.  Then these $\EM$-diagrams must be in one of three relative positions:
\begin{itemize}
\item $k<i'$, 
\item $j<i'<k'<k$, or
\item $i<i'<k'<j$.
\end{itemize}
These are the three noncrossing $\EM$-diagrams from $3 \times 2$ tableaux.  This completes the proof.
\end{proof}

\section{From $\EM$-diagrams to webs}\label{section: m to web}
In this section, we restrict to the case of Young diagrams with three rows.  In this section, we describe how to transform an $\EM$-diagram for a three-row Young tableau into a web for $\mathfrak{sl}_3$.  We then prove that the webs we obtain are {\em irreducible}.  

Recall from the introduction that a web for $\mathfrak{sl}_3$ is a planar directed graph with boundary so that (1) internal vertices are trivalent and boundary vertices have degree one, and (2) each vertex is either a source or a sink.  Webs were originally defined to be embedded in a disk.  However, for convenience, we cut the disk to create a boundary line.

\subsection{Constructing webs from $\EM$-diagrams} \label{section: resolved m definition}
The $\EM$-diagrams obtained in the previous section are almost, but not quite, webs for $\mathfrak{sl}_3$.  There are three problems:
\begin{enumerate}
\item the boundary vertex $j$ on an $\EM$ given by $(i,j,k)$ has degree two; 
\item the edges are undirected; and
\item there are degree-four vertices where two arcs cross.
\end{enumerate}
Each problem is easily addressed, so easily that in practice we often assume that an $\EM$-diagram has already had the next steps performed.
\begin{enumerate}
\item Each degree-two boundary vertex should be replaced with the shape $Y$.  Hence each $\EM$ has a unique trivalent vertex, which we think of as the intersection of its two arcs.  (One might say that these are $\EM$-diagrams rather than $n n$-diagrams.)
\item Edges should be directed so that:
\begin{enumerate}
\item The edges in each $\EM$ are directed away from the boundary and towards the trivalent vertex on the $\EM$.  (The direction of each edge remains the same across any intersections with other $\EM$s.)
\item The edges in an isolated arc should be directed from the boundary vertex on the bottom row of the Young tableau to the boundary vertex on the middle row of the tableau.
\end{enumerate}
\end{enumerate}

Given these conventions, there is a unique way to make the degree-four vertex at the intersection of two arcs trivalent.  We call this process {\em resolving} the diagram, and describe it in the next lemma.

\begin{lemma} \label{lemma:trivalization}
Let $v$ be a 4-valent interior vertex in an $\EM$-diagram.  There is a unique way to replace $v$ with a pair of trivalent vertices so that the $\EM$-diagram is unchanged outside of a small neighborhood of $v$.
\end{lemma}

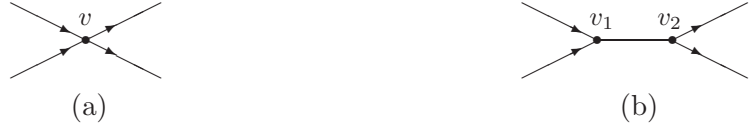
\begin{figure}[h]
\hspace{0.4in}
\begin{picture}(40,20)(0,0)
\put(5,5){\line(2,1){20}}
\put(11,8){\vector(2,1){2}}
\put(17,11){\vector(2,1){2}}
\put(5,15){\line(2,-1){20}}
\put(11,12){\vector(2,-1){2}}
\put(17,9){\vector(2,-1){2}}
\put(15, 10){\circle*{1}}
\put(14, 12){\small{$v$}}
\put(13,0){(a)}
\end{picture}
\hspace{1in}
\begin{picture}(40,20)(0,0)
\put(5,15){\line(2,-1){10}}
\put(11,8){\vector(2,1){2}}
\put(11,12){\vector(2,-1){2}}
\put(5,5){\line(2,1){10}}
\put(15, 10){\circle*{1}}
\put(14, 12){\small{$v_1$}}
\put(15,10){\line(1,0){10}}
\put(25, 10){\circle*{1}}
\put(23,12){\small{$v_2$}}
\put(25,10){\line(2,1){10}}
\put(27,11){\vector(2,1){2}}
\put(27,9){\vector(2,-1){2}}
\put(25,10){\line(2,-1){10}}
\put(18,0){(b)}
\end{picture}
\caption{Trivalizing vertices}\label{trivalizing vertices}
\end{figure}

\begin{proof}
A four-valent vertex $v$ in an $\EM$-diagram occurs when two directed arcs cross, which looks locally like Figure \ref{trivalizing vertices} (a).  Each arc is directed, so there are two edges incident to $v$ that are directed in and two that are directed out.  Vertices in the web must be trivalent with incident edges all directed in or all directed out.  There is one way to partition the edges incident to $v$ into in-edges and out-edges.  Create a new vertex $v_1$ incident to the in-edges and a new vertex $v_2$ incident to the out-edges.  An edge between $v_1$ and $v_2$ makes both vertices trivalent, and must be directed $v_2 \mapsto v_1$ to satisfy the conditions of the web.  Locally this creates the diagram in Figure \ref{trivalizing vertices} (b).
\end{proof}

\subsection{Examples} \label{section: web examples}

Extending Section \ref{section: m-diagram examples}, we give the web corresponding to each $\EM$-diagram for the standard tableaux of shape $3 \times 2$.  (Each web is a planar graph; each graph here is a reasonably symmetric example from its isomorphism class.)

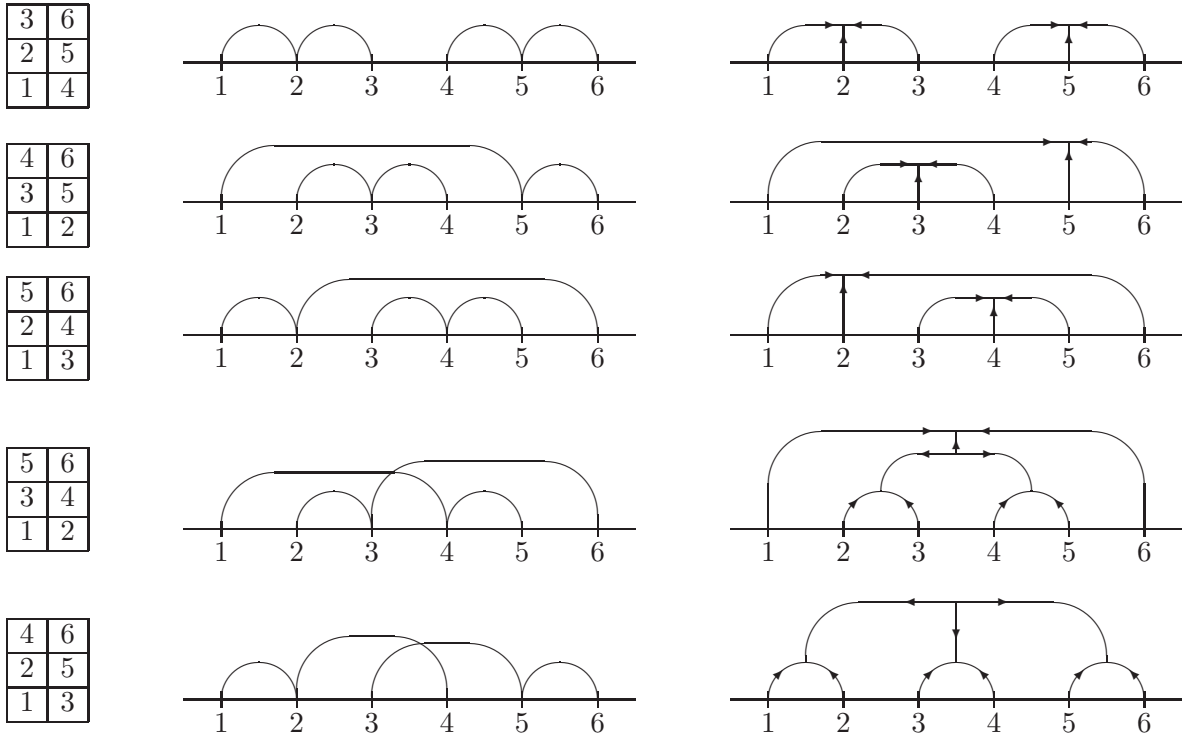
\[
\begin{array}{|c|c|}
\cline{1-2}
3 & 6\\
\cline{1-2}
2 & 5\\
\cline{1-2}
1 & 4\\
\cline{1-2}
\end{array} 
\hspace{0.5in} 
\begin{picture}(60, 10)(0, 5)
\put(0,5){\line(1,0){60}}
\put(5,4){\line(0,1){2}}
\put(15,4){\line(0,1){2}}
\put(25,4){\line(0,1){2}}
\put(35,4){\line(0,1){2}}
\put(45,4){\line(0,1){2}}
\put(55,4){\line(0,1){2}}
\put(5,2){\makebox(0,0){$1$}}
\put(15,2){\makebox(0,0){$2$}}
\put(25,2){\makebox(0,0){$3$}}
\put(35,2){\makebox(0,0){$4$}}
\put(45,2){\makebox(0,0){$5$}}
\put(55,2){\makebox(0,0){$6$}}
\put(10,5){\oval(10,10)[t]}
\put(40,5){\oval(10,10)[t]}
\put(20,5){\oval(10,10)[t]}
\put(50,5){\oval(10,10)[t]}
\end{picture}
\hspace{0.5in} 
\begin{picture}(60, 10)(0, 5)
\put(0,5){\line(1,0){60}}
\put(5,4){\line(0,1){2}}
\put(15,4){\line(0,1){2}}
\put(25,4){\line(0,1){2}}
\put(35,4){\line(0,1){2}}
\put(45,4){\line(0,1){2}}
\put(55,4){\line(0,1){2}}
\put(5,2){\makebox(0,0){$1$}}
\put(15,2){\makebox(0,0){$2$}}
\put(25,2){\makebox(0,0){$3$}}
\put(35,2){\makebox(0,0){$4$}}
\put(45,2){\makebox(0,0){$5$}}
\put(55,2){\makebox(0,0){$6$}}
\put(15,5){\oval(20,10)[t]}
\put(45,5){\oval(20,10)[t]}
\put(15,5){\line(0,1){5}}
\put(45,5){\line(0,1){5}}
\put(15,7){\vector(0,1){2}}
\put(12,10){\vector(1,0){2}}
\put(18,10){\vector(-1,0){2}}
\put(45,7){\vector(0,1){2}}
\put(42,10){\vector(1,0){2}}
\put(48,10){\vector(-1,0){2}}
\end{picture}
\]
\[\begin{array}{|c|c|}
\cline{1-2}
4 & 6\\
\cline{1-2}
3 & 5\\
\cline{1-2}
1 & 2\\
\cline{1-2}
\end{array}
\hspace{0.5in} 
\begin{picture}(60, 10)(0, 5)
\put(0,5){\line(1,0){60}}
\put(5,4){\line(0,1){2}}
\put(15,4){\line(0,1){2}}
\put(25,4){\line(0,1){2}}
\put(35,4){\line(0,1){2}}
\put(45,4){\line(0,1){2}}
\put(55,4){\line(0,1){2}}
\put(5,2){\makebox(0,0){$1$}}
\put(15,2){\makebox(0,0){$2$}}
\put(25,2){\makebox(0,0){$3$}}
\put(35,2){\makebox(0,0){$4$}}
\put(45,2){\makebox(0,0){$5$}}
\put(55,2){\makebox(0,0){$6$}}
\put(20,5){\oval(10,10)[t]}
\put(30,5){\oval(10,10)[t]}
\put(25,5){\oval(40,15)[t]}
\put(50,5){\oval(10,10)[t]}
\end{picture}
\hspace{0.5in} 
\begin{picture}(60, 10)(0, 5)
\put(0,5){\line(1,0){60}}
\put(5,4){\line(0,1){2}}
\put(15,4){\line(0,1){2}}
\put(25,4){\line(0,1){2}}
\put(35,4){\line(0,1){2}}
\put(45,4){\line(0,1){2}}
\put(55,4){\line(0,1){2}}
\put(5,2){\makebox(0,0){$1$}}
\put(15,2){\makebox(0,0){$2$}}
\put(25,2){\makebox(0,0){$3$}}
\put(35,2){\makebox(0,0){$4$}}
\put(45,2){\makebox(0,0){$5$}}
\put(55,2){\makebox(0,0){$6$}}
\put(25,5){\oval(20,10)[t]}
\put(25,5){\line(0,1){5}}
\put(30,5){\oval(50,16)[t]}
\put(45,5){\line(0,1){8}}
\put(25,7){\vector(0,1){2}}
\put(22,10){\vector(1,0){2}}
\put(28,10){\vector(-1,0){2}}
\put(45,10){\vector(0,1){2}}
\put(41,13){\vector(1,0){2}}
\put(48,13){\vector(-1,0){2}}
\end{picture}
\]
\[\begin{array}{|c|c|}
\cline{1-2}
5 & 6\\
\cline{1-2}
2 & 4\\
\cline{1-2}
1 & 3\\
\cline{1-2}
\end{array}
\hspace{0.5in} 
\begin{picture}(60, 10)(0, 5)
\put(0,5){\line(1,0){60}}
\put(5,4){\line(0,1){2}}
\put(15,4){\line(0,1){2}}
\put(25,4){\line(0,1){2}}
\put(35,4){\line(0,1){2}}
\put(45,4){\line(0,1){2}}
\put(55,4){\line(0,1){2}}
\put(5,2){\makebox(0,0){$1$}}
\put(15,2){\makebox(0,0){$2$}}
\put(25,2){\makebox(0,0){$3$}}
\put(35,2){\makebox(0,0){$4$}}
\put(45,2){\makebox(0,0){$5$}}
\put(55,2){\makebox(0,0){$6$}}
\put(30,5){\oval(10,10)[t]}
\put(40,5){\oval(10,10)[t]}
\put(35,5){\oval(40,15)[t]}
\put(10,5){\oval(10,10)[t]}
\end{picture}
\hspace{0.5in} 
\begin{picture}(60, 10)(0, 5)
\put(0,5){\line(1,0){60}}
\put(5,4){\line(0,1){2}}
\put(15,4){\line(0,1){2}}
\put(25,4){\line(0,1){2}}
\put(35,4){\line(0,1){2}}
\put(45,4){\line(0,1){2}}
\put(55,4){\line(0,1){2}}
\put(5,2){\makebox(0,0){$1$}}
\put(15,2){\makebox(0,0){$2$}}
\put(25,2){\makebox(0,0){$3$}}
\put(35,2){\makebox(0,0){$4$}}
\put(45,2){\makebox(0,0){$5$}}
\put(55,2){\makebox(0,0){$6$}}
\put(35,5){\oval(20,10)[t]}
\put(35,5){\line(0,1){5}}
\put(30,5){\oval(50,16)[t]}
\put(15,5){\line(0,1){8}}
\put(35,7){\vector(0,1){2}}
\put(32,10){\vector(1,0){2}}
\put(38,10){\vector(-1,0){2}}
\put(15,10){\vector(0,1){2}}
\put(12,13){\vector(1,0){2}}
\put(19,13){\vector(-1,0){2}}
\end{picture}
\]
\[\begin{array}{|c|c|}
\cline{1-2}
5 & 6\\
\cline{1-2}
3 & 4\\
\cline{1-2}
1 & 2\\
\cline{1-2}
\end{array}
\hspace{0.5in} 
\begin{picture}(60, 10)(0, 8)
\put(0,5){\line(1,0){60}}
\put(5,4){\line(0,1){2}}
\put(15,4){\line(0,1){2}}
\put(25,4){\line(0,1){2}}
\put(35,4){\line(0,1){2}}
\put(45,4){\line(0,1){2}}
\put(55,4){\line(0,1){2}}
\put(5,2){\makebox(0,0){$1$}}
\put(15,2){\makebox(0,0){$2$}}
\put(25,2){\makebox(0,0){$3$}}
\put(35,2){\makebox(0,0){$4$}}
\put(45,2){\makebox(0,0){$5$}}
\put(55,2){\makebox(0,0){$6$}}
\put(20,5){\oval(10,10)[t]}
\put(40,5){\oval(10,10)[t]}
\put(40,5){\oval(30,18)[t]}
\put(20,5){\oval(30,15)[t]}
\end{picture}
\hspace{0.5in} 
\begin{picture}(60, 15)(0, 8)
\put(0,5){\line(1,0){60}}
\put(5,4){\line(0,1){2}}
\put(15,4){\line(0,1){2}}
\put(25,4){\line(0,1){2}}
\put(35,4){\line(0,1){2}}
\put(45,4){\line(0,1){2}}
\put(55,4){\line(0,1){2}}
\put(5,2){\makebox(0,0){$1$}}
\put(15,2){\makebox(0,0){$2$}}
\put(25,2){\makebox(0,0){$3$}}
\put(35,2){\makebox(0,0){$4$}}
\put(45,2){\makebox(0,0){$5$}}
\put(55,2){\makebox(0,0){$6$}}
\put(20,5){\oval(10,10)[t]}
\put(40,5){\oval(10,10)[t]}
\put(30,10){\oval(20,10)[t]}
\put(30,5){\oval(50,26)[t]}
\put(30,15){\line(0,1){3}}
\put(25,18){\vector(1,0){2}}
\put(35,18){\vector(-1,0){2}}
\put(30,15){\vector(0,1){2}}
\put(27,15){\vector(-1,0){2}}
\put(33,15){\vector(1,0){2}}
\put(15,7){\vector(1,1){2}}
\put(25,7){\vector(-1,1){2}}
\put(35,7){\vector(1,1){2}}
\put(45,7){\vector(-1,1){2}}
\end{picture}
\]
\[\begin{array}{|c|c|}
\cline{1-2}
4 & 6\\
\cline{1-2}
2 & 5\\
\cline{1-2}
1 & 3\\
\cline{1-2}
\end{array}
\hspace{0.5in} 
\begin{picture}(60, 10)(0, 8)
\put(0,5){\line(1,0){60}}
\put(5,4){\line(0,1){2}}
\put(15,4){\line(0,1){2}}
\put(25,4){\line(0,1){2}}
\put(35,4){\line(0,1){2}}
\put(45,4){\line(0,1){2}}
\put(55,4){\line(0,1){2}}
\put(5,2){\makebox(0,0){$1$}}
\put(15,2){\makebox(0,0){$2$}}
\put(25,2){\makebox(0,0){$3$}}
\put(35,2){\makebox(0,0){$4$}}
\put(45,2){\makebox(0,0){$5$}}
\put(55,2){\makebox(0,0){$6$}}
\put(10,5){\oval(10,10)[t]}
\put(50,5){\oval(10,10)[t]}
\put(25,5){\oval(20,17)[t]}
\put(35,5){\oval(20,15)[t]}
\end{picture}
\hspace{0.5in} 
\begin{picture}(60, 15)(0, 8)
\put(0,5){\line(1,0){60}}
\put(5,4){\line(0,1){2}}
\put(15,4){\line(0,1){2}}
\put(25,4){\line(0,1){2}}
\put(35,4){\line(0,1){2}}
\put(45,4){\line(0,1){2}}
\put(55,4){\line(0,1){2}}
\put(5,2){\makebox(0,0){$1$}}
\put(15,2){\makebox(0,0){$2$}}
\put(25,2){\makebox(0,0){$3$}}
\put(35,2){\makebox(0,0){$4$}}
\put(45,2){\makebox(0,0){$5$}}
\put(55,2){\makebox(0,0){$6$}}
\put(10,5){\oval(10,10)[t]}
\put(30,5){\oval(10,10)[t]}
\put(50,5){\oval(10,10)[t]}
\put(30,10){\oval(40,16)[t]}
\put(30,10){\line(0,1){8}}
\put(5,7){\vector(1,1){2}}
\put(15,7){\vector(-1,1){2}}
\put(25,7){\vector(1,1){2}}
\put(35,7){\vector(-1,1){2}}
\put(45,7){\vector(1,1){2}}
\put(55,7){\vector(-1,1){2}}
\put(25,18){\vector(-1,0){2}}
\put(35,18){\vector(1,0){2}}
\put(30,15){\vector(0,-1){2}}
\end{picture}
\]

\subsection{Webs obtained from $\EM$-diagrams are reduced} \label{section: reduced webs}
Resolved $\EM$-diagrams are webs by construction: they are planar directed graphs; each internal vertex is trivalent and each vertex on the boundary line has degree one; and each vertex is a source or a sink (since each trivalent vertex was constructed to have either all incident edges directed in or all incident edges directed out).  In fact, a stronger condition holds.

\begin{definition}
A web in the $\mathfrak{sl}_3$-spider is reduced if it is non-elliptic, namely each (interior) face has at least six edges on its boundary.
\end{definition}

The following series of small lemmas proves that the webs obtained from $\EM$-diagrams are reduced.

\begin{lemma}
Let $T$ be a Young tableau of arbitrary shape.  The web obtained by resolving the $\EM$-diagram for $T$ has no face with two edges on its boundary.
\end{lemma}

\begin{proof}
A face with two edges is bounded by two arcs which cross each other twice.  This does not happen in $\EM$-diagrams, by Proposition \ref{proposition:basic noncrossing props}.
\end{proof}

\begin{lemma}
Let $T$ be a Young tableau of arbitrary shape.  The web obtained by resolving the $\EM$-diagram for $T$ has no face with an odd number of edges on its boundary.
\end{lemma}

\begin{proof}
Each edge is directed, and each vertex is either a sink (all edges are oriented in) or a source (all edges are oriented out).  Hence the edges in any (undirected) cycle alternate orientations, and so every cycle in the graph must have an even number of edges.
\end{proof}

\begin{lemma}
The web obtained by resolving the $\EM$-diagram for a three-row Young tableau $T$ has no interior face with exactly four edges on its boundary.
\end{lemma}

\begin{proof}
{\em All $\EM$-diagrams in this proof correspond to three-row Young tableaux.}  The proof is by contradiction: we locally reconstruct arcs that could produce a face with four edges, and then prove that $\EM$-diagrams contain no such arrangement of arcs.  The previous lemma showed that the edges in each cycle in an $\EM$-diagram alternate orientation, as in Figure \ref{figure:simple square}.

\begin{figure}[h]
\begin{picture}(40,20)(-10,-2)
\put(0,0){\circle*{1}}
\put(0,15){\circle*{1}}
\put(15,0){\circle*{1}}
\put(15,15){\circle*{1}}

\put(1,2){$v_1$}
\put(1,12){$v_2$}
\put(11,12){$v_3$}
\put(11,2){$v_4$}

\put(0,0){\vector(0,1){8}}
\put(0,0){\line(0,1){15}}
\put(0,0){\vector(1,0){8}}
\put(0,0){\line(1,0){15}}

\put(15,15){\vector(0,-1){7}}
\put(15,15){\line(0,-1){15}}
\put(15,15){\vector(-1,0){7}}
\put(15,15){\line(-1,0){15}}

\put(0,0){\vector(-1,-1){5}}
\put(15,15){\vector(1,1){5}}
\put(-5,20){\line(1,-1){5}}
\put(20,-5){\line(-1,1){5}}
\end{picture}

\caption{Face with four edges}\label{figure:simple square}
\end{figure}
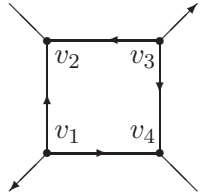

We first show that none of the edges of the square were added while resolving the $\EM$-diagram.  Assume otherwise.  After resolving a vertex as in Lemma \ref{lemma:trivalization}, the edges associated to a single arc are on different faces.  Hence an arc that enters the square at $v_2$ cannot leave from either $v_1$ or $v_3$ (and similarly for arcs entering at $v_4$).  Without loss of generality we conclude that an arc that enters at $v_2$ must proceed along the edge $v_3v_4$, as in Figure \ref{figure:resolving in square}.  
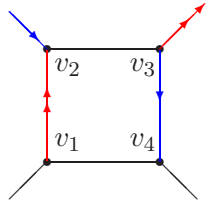
\begin{figure}[h]
\begin{picture}(40,20)(-10,-2)
\put(0,0){\circle*{1}}
\put(0,15){\circle*{1}}
\put(15,0){\circle*{1}}
\put(15,15){\circle*{1}}

\put(1,2){$v_1$}
\put(1,12){$v_2$}
\put(11,12){$v_3$}
\put(11,2){$v_4$}

\put(0,0){\color{red}{\vector(0,1){8}}}
\put(0,2){\color{red}{\vector(0,1){8}}}
\put(0,0){\color{red}{\line(0,1){15}}}
\put(0,0){\line(1,0){15}}

\put(15,15){\color{blue}{\vector(0,-1){7}}}
\put(15,15){\color{blue}{\line(0,-1){15}}}
\put(15,15){\line(-1,0){15}}

\put(0,0){\line(-1,-1){5}}
\put(15,15){\color{red}{\vector(1,1){4}}}
\put(16,16){\color{red}{\vector(1,1){5}}}
\put(-5,20){\color{blue}{\line(1,-1){5}}}
\put(-5,20){\color{blue}{\vector(1,-1){4}}}
\put(20,-5){\line(-1,1){5}}
\end{picture}

\caption{Face with edge added during the trivializing process}\label{figure:resolving in square}
\end{figure}
Arcs in an $\EM$-diagram cross at most once by Proposition \ref{proposition:basic noncrossing props}.  We conclude that the arc $v_3v_4$ cannot cross the arc $v_1v_2$, and so leaves the face at $v_4$ rather than $v_1$.  Hence there is a third arc that must be incident to both $v_1$ and $v_4$.  Each pair of these three arcs cross (at $v_1$, $v_2$, or $v_4$ respectively).  Each arc in the $\EM$-diagram of a three-row tableau is a first or second arc.  First arcs can only cross second arcs by Proposition \ref{proposition:basic noncrossing props}, so this configuration of arcs does not come from an $\EM$-diagram.

All edges incident to the original trivalent vertex on an $\EM$ are oriented inwards, so at most two of the vertices in the four-cycle are the original trivalent vertex on an $\EM$.  If both $v_2$ and $v_4$ are vertices from the $\EM$-diagram, then two $\EM$s cross twice (namely at $v_1$ and $v_3$), contradicting Corollary \ref{corollary:two ms relative}.  

Suppose at most one of the vertices is a vertex from the original $\EM$-diagram, without loss of generality $v_4$.  Two edges from the same arc cannot bound the same face after resolving as in Lemma \ref{lemma:trivalization}, so the four arcs bounding the square must have the relative positions shown in Figure \ref{figure:arcs on square}.  We mark these arcs with one, two, three, and four arrows, and 
\begin{figure}[h]
\begin{picture}(40,29)(-10,-6)
\put(0,0){\circle*{1}}
\put(0,15){\circle*{1}}
\put(15,0){\circle*{1}}
\put(15,15){\circle*{1}}

\put(1,2){$v_1$}
\put(1,12){$v_2$}
\put(11,12){$v_3$}
\put(11,2){$v_4$}

\put(0,0){\color{red}{\vector(0,1){8}}}
\put(0,2){\color{red}{\vector(0,1){8}}}
\put(0,0){\color{red}{\line(0,1){15}}}
\put(0,0){\color{green}{\vector(1,0){6}}}
\put(2,0){\color{green}{\vector(1,0){6}}}
\put(4,0){\color{green}{\vector(1,0){6}}}
\put(0,0){\color{green}{\line(1,0){15}}}

\put(15,15){\color{blue}{\vector(0,-1){7}}}
\put(15,15){\color{blue}{\line(0,-1){15}}}
\put(15,15){\vector(-1,0){5}}
\put(13,15){\vector(-1,0){5}}
\put(11,15){\vector(-1,0){5}}
\put(9,15){\vector(-1,0){5}}
\put(15,15){\line(-1,0){15}}

\put(0,0){\line(-1,-1){5}}
\put(15,15){\line(1,1){5}}
\put(-5,20){\line(1,-1){5}}
\put(20,-5){\line(-1,1){5}}

\put(-12,-5){\color{green}{\vector(1,0){3}}}
\put(-10,-5){\color{green}{\vector(1,0){3}}}
\put(-8,-5){\color{green}{\vector(1,0){3}}}

\put(20,-5){\color{green}{\vector(1,0){3}}}
\put(22,-5){\color{green}{\vector(1,0){3}}}
\put(24,-5){\color{green}{\vector(1,0){3}}}

\put(-5,-11){\color{red}{\vector(0,1){3}}}
\put(-5,-9){\color{red}{\vector(0,1){3}}}
\put(-5,-10){\color{red}{\line(0,1){5}}}

\put(-5,20){\color{red}{\vector(0,1){3}}}
\put(-5,22){\color{red}{\vector(0,1){3}}}
\put(-5,20){\color{red}{\line(0,1){5}}}

\put(26,20){\vector(-1,0){3}}
\put(25,20){\vector(-1,0){3}}
\put(24,20){\vector(-1,0){3}}
\put(23,20){\vector(-1,0){3}}

\put(-5,20){\vector(-1,0){3}}
\put(-6,20){\vector(-1,0){3}}
\put(-7,20){\vector(-1,0){3}}
\put(-8,20){\vector(-1,0){3}}

\put(20,25){\color{blue}{\vector(0,-1){3}}}
\put(20,25){\color{blue}{\line(0,-1){5}}}

\put(20,-5){\color{blue}{\vector(0,-1){3}}}
\put(20,-5){\color{blue}{\line(0,-1){5}}}

\end{picture}

\caption{Arcs from original $\EM$-diagram}\label{figure:arcs on square}
\end{figure}
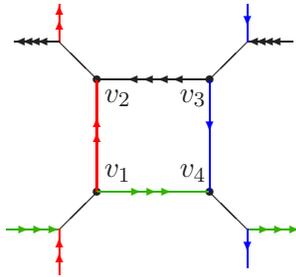
call them 1,2,3,4, respectively.   (If $v_4$ is a vertex from the original $\EM$-diagram, then arcs 1 and 3 do not continue out of the square in the bottom right of Figure \ref{figure:arcs on square}.)  

First arcs and second arcs do not cross, so we conclude from the interior square that either 3, 4 are both first arcs and 1, 2 are both second arcs, or vice versa.  Two arcs can cross at most once, and $i^{th}$ arcs are noncrossing, so the boundary vertices of arcs 1,2,3,4 are in the same relative position as in the perimeter of Figure \ref{figure:arcs on square}.  (If $v_4$ is a vertex from the original $\EM$-diagram, then in addition neither arc 1 nor arc 3 is the edge in the middle of the $\EM$, since the second edge of an $\EM$ in a resolved $\EM$-diagram crosses nothing.)

By construction the second arcs in an $\EM$-diagram are directed from right to left and first arcs are directed from left to right, regardless of whether the first arc is isolated or not.  Consider just the arcs 1,2,3,4.  Regardless of which endpoint is leftmost on the boundary line, two of the leftmost three of these arcs cross {\em and are directed the same way}.  This contradicts the fact that arcs of the same type (either first or second) are noncrossing.  (If $v_4$ is a vertex from the original $\EM$-diagram, we may also use the fact that arc 1 cannot be the edge in the middle of the $\EM$ containing $v_4$ to obtain a contradiction.)

We conclude that the resolution of an $\EM$-diagram for a three-row Young tableau has no interior face whose boundary has four edges.
\end{proof}

\section{The map from irreducible webs to three-row Young tableaux}

Together, the previous two sections give a map from three-row Young tableaux to irreducible webs for $\mathfrak{sl}_3$.  In this section, we prove that the map is injective.  To do this, we modify a map of Khovanov-Kuperberg that we call the {\em depth map}.  They defined the depth of a face to be the distance from the unbounded face in the planar graph dual to the web; in other words, it is the minimal number of edges crossed by paths between a given face and the unbounded face.  We show that their depth is the same as the depth measured by the number of semicircles in an $\EM$-diagram that contain a face.  We then use this to give an elementary proof that depth inverts the map that sends a standard three-row Young tableau to its resolved $\EM$-diagram.

\subsection{Circle and path depth}\label{subsection: different depths}

Fix a horizontal line $\ell$.  Generalizing $\EM$-diagrams simplifies proofs in this section.

\begin{definition}
A {\em sequence of upper semicircles ${\mathcal C}$} is a set of upper semicircles, each centered on a different point along the line $\ell$.   
\end{definition}

The intersections of the semicircles in ${\mathcal C}$ naturally define a planar graph. 

\begin{definition}
The sequence of upper semicircles ${\mathcal C}$ determines a planar graph $G_{\mathcal C}$ as follows:
\begin{itemize}
\item each point of intersection between two semicircles or between a semicircle and the line $\ell$ is a vertex, and 
\item each arc between vertices is an edge.  
\end{itemize}
\end{definition}

Colloquially, the circle depth of a point on a face of this graph is the number of semicircles containing the point.

\begin{definition}
Let $x$ be a point on a face of the planar graph $G_{\mathcal C}$ determined by the sequence of upper semicircles ${\mathcal C}$.  The {\em circle depth} of $p$, denoted $d^c(x,\mathcal{C})$,
is the number of semicircles above $x$.
\end{definition}

Any planar graph that lies above a horizontal line has another natural definition of depth, which we call path depth.  In the context of spiders, it was defined by Khovanov-Kuperberg \cite{KK99}.

\begin{definition}
Given a planar graph $G$ that lies above a horizontal boundary line, let $f_0$ be the unbounded face above the line.  Let $x$ be any point on the interior of a face of $G$.  The {\em path depth} of $x$, denoted $d^p(x,G)$, is the minimal number of edges crossed by any path from $x$ to $f_0$ that does not cross the boundary.  (The path from $x$ to $f_0$ must cross at the interior of edges and not at vertices.)
\end{definition}

Equivalently the path is an ordinary path in the dual graph to the planar graph $G$.  In particular, the path depth of $x$ is independent of small deformations in the planar representation of $G$, for instance stretching, contracting, or rotating edges.  

The planar graphs that we consider are resolutions of the planar graphs $G_{\mathcal C}$ corresponding to a set of upper semicircles ${\mathcal C}$.  We abuse notation and write $d^p(x,\mathcal{C})$ in this case.

Both circle depth and path depth are constant on each face of a planar graph, since any two points in the same face can be connected by a path that does not cross any edges of the graph.  We will compare depth for different sets ${\mathcal C}$, so it is more convenient to consider depth as a function on points rather than faces.

\subsection{Examples} \label{section: depth examples}

The following examples show circle depth for each $\EM$-diagram and path depth for each web for the standard tableaux of shape $3 \times 2$.  In all cases, the unbounded face has depth zero.  The reader may notice that circle depth for an $\EM$-diagram is quite similar to path depth for the corresponding web; we prove they are the same in the next section.  
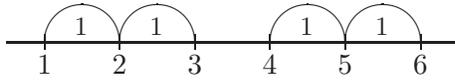
\[
\begin{array}{|c|c|}
\cline{1-2}
3 & 6\\
\cline{1-2}
2 & 5\\
\cline{1-2}
1 & 4\\
\cline{1-2}
\end{array} 
\hspace{0.5in} 
\begin{picture}(60, 10)(0, 5)
\put(0,5){\line(1,0){60}}
\put(5,4){\line(0,1){2}}
\put(15,4){\line(0,1){2}}
\put(25,4){\line(0,1){2}}
\put(35,4){\line(0,1){2}}
\put(45,4){\line(0,1){2}}
\put(55,4){\line(0,1){2}}
\put(5,2){\makebox(0,0){$1$}}
\put(15,2){\makebox(0,0){$2$}}
\put(25,2){\makebox(0,0){$3$}}
\put(35,2){\makebox(0,0){$4$}}
\put(45,2){\makebox(0,0){$5$}}
\put(55,2){\makebox(0,0){$6$}}
\put(10,5){\oval(10,10)[t]}
\put(40,5){\oval(10,10)[t]}
\put(20,5){\oval(10,10)[t]}
\put(50,5){\oval(10,10)[t]}
\put(9,6){\small 1}
\put(19,6){\small 1}
\put(39,6){\small 1}
\put(49,6){\small 1}
\end{picture}
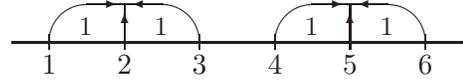
\hspace{0.5in} 
\begin{picture}(60, 10)(0, 5)
\put(0,5){\line(1,0){60}}
\put(5,4){\line(0,1){2}}
\put(15,4){\line(0,1){2}}
\put(25,4){\line(0,1){2}}
\put(35,4){\line(0,1){2}}
\put(45,4){\line(0,1){2}}
\put(55,4){\line(0,1){2}}
\put(5,2){\makebox(0,0){$1$}}
\put(15,2){\makebox(0,0){$2$}}
\put(25,2){\makebox(0,0){$3$}}
\put(35,2){\makebox(0,0){$4$}}
\put(45,2){\makebox(0,0){$5$}}
\put(55,2){\makebox(0,0){$6$}}
\put(15,5){\oval(20,10)[t]}
\put(45,5){\oval(20,10)[t]}
\put(15,5){\line(0,1){5}}
\put(45,5){\line(0,1){5}}
\put(15,7){\vector(0,1){2}}
\put(12,10){\vector(1,0){2}}
\put(18,10){\vector(-1,0){2}}
\put(45,7){\vector(0,1){2}}
\put(42,10){\vector(1,0){2}}
\put(48,10){\vector(-1,0){2}}
\put(9,6){\small 1}
\put(19,6){\small 1}
\put(39,6){\small 1}
\put(49,6){\small 1}
\end{picture}
\]
\[\begin{array}{|c|c|}
\cline{1-2}
4 & 6\\
\cline{1-2}
3 & 5\\
\cline{1-2}
1 & 2\\
\cline{1-2}
\end{array}
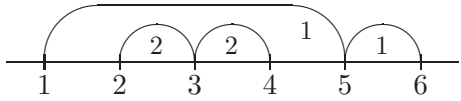
\hspace{0.5in} 
\begin{picture}(60, 10)(0, 5)
\put(0,5){\line(1,0){60}}
\put(5,4){\line(0,1){2}}
\put(15,4){\line(0,1){2}}
\put(25,4){\line(0,1){2}}
\put(35,4){\line(0,1){2}}
\put(45,4){\line(0,1){2}}
\put(55,4){\line(0,1){2}}
\put(5,2){\makebox(0,0){$1$}}
\put(15,2){\makebox(0,0){$2$}}
\put(25,2){\makebox(0,0){$3$}}
\put(35,2){\makebox(0,0){$4$}}
\put(45,2){\makebox(0,0){$5$}}
\put(55,2){\makebox(0,0){$6$}}
\put(20,5){\oval(10,10)[t]}
\put(30,5){\oval(10,10)[t]}
\put(25,5){\oval(40,15)[t]}
\put(50,5){\oval(10,10)[t]}
\put(19,6){\small 2}
\put(29,6){\small 2}
\put(49,6){\small 1}
\put(39,8){\small 1}
\end{picture}
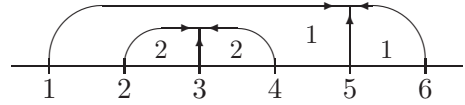
\hspace{0.5in} 
\begin{picture}(60, 10)(0, 5)
\put(0,5){\line(1,0){60}}
\put(5,4){\line(0,1){2}}
\put(15,4){\line(0,1){2}}
\put(25,4){\line(0,1){2}}
\put(35,4){\line(0,1){2}}
\put(45,4){\line(0,1){2}}
\put(55,4){\line(0,1){2}}
\put(5,2){\makebox(0,0){$1$}}
\put(15,2){\makebox(0,0){$2$}}
\put(25,2){\makebox(0,0){$3$}}
\put(35,2){\makebox(0,0){$4$}}
\put(45,2){\makebox(0,0){$5$}}
\put(55,2){\makebox(0,0){$6$}}
\put(25,5){\oval(20,10)[t]}
\put(25,5){\line(0,1){5}}
\put(30,5){\oval(50,16)[t]}
\put(45,5){\line(0,1){8}}
\put(25,7){\vector(0,1){2}}
\put(22,10){\vector(1,0){2}}
\put(28,10){\vector(-1,0){2}}
\put(45,10){\vector(0,1){2}}
\put(41,13){\vector(1,0){2}}
\put(48,13){\vector(-1,0){2}}
\put(19,6){\small 2}
\put(29,6){\small 2}
\put(49,6){\small 1}
\put(39,8){\small 1}
\end{picture}
\]
\[\begin{array}{|c|c|}
\cline{1-2}
5 & 6\\
\cline{1-2}
2 & 4\\
\cline{1-2}
1 & 3\\
\cline{1-2}
\end{array}
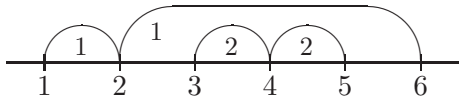
\hspace{0.5in} 
\begin{picture}(60, 10)(0, 5)
\put(0,5){\line(1,0){60}}
\put(5,4){\line(0,1){2}}
\put(15,4){\line(0,1){2}}
\put(25,4){\line(0,1){2}}
\put(35,4){\line(0,1){2}}
\put(45,4){\line(0,1){2}}
\put(55,4){\line(0,1){2}}
\put(5,2){\makebox(0,0){$1$}}
\put(15,2){\makebox(0,0){$2$}}
\put(25,2){\makebox(0,0){$3$}}
\put(35,2){\makebox(0,0){$4$}}
\put(45,2){\makebox(0,0){$5$}}
\put(55,2){\makebox(0,0){$6$}}
\put(30,5){\oval(10,10)[t]}
\put(40,5){\oval(10,10)[t]}
\put(35,5){\oval(40,15)[t]}
\put(10,5){\oval(10,10)[t]}
\put(9,6){\small 1}
\put(29,6){\small 2}
\put(39,6){\small 2}
\put(19,8){\small 1}
\end{picture}
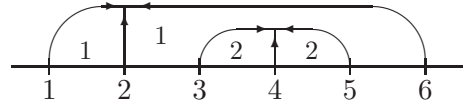
\hspace{0.5in} 
\begin{picture}(60, 10)(0, 5)
\put(0,5){\line(1,0){60}}
\put(5,4){\line(0,1){2}}
\put(15,4){\line(0,1){2}}
\put(25,4){\line(0,1){2}}
\put(35,4){\line(0,1){2}}
\put(45,4){\line(0,1){2}}
\put(55,4){\line(0,1){2}}
\put(5,2){\makebox(0,0){$1$}}
\put(15,2){\makebox(0,0){$2$}}
\put(25,2){\makebox(0,0){$3$}}
\put(35,2){\makebox(0,0){$4$}}
\put(45,2){\makebox(0,0){$5$}}
\put(55,2){\makebox(0,0){$6$}}
\put(35,5){\oval(20,10)[t]}
\put(35,5){\line(0,1){5}}
\put(30,5){\oval(50,16)[t]}
\put(15,5){\line(0,1){8}}
\put(35,7){\vector(0,1){2}}
\put(32,10){\vector(1,0){2}}
\put(38,10){\vector(-1,0){2}}
\put(15,10){\vector(0,1){2}}
\put(12,13){\vector(1,0){2}}
\put(19,13){\vector(-1,0){2}}
\put(9,6){\small 1}
\put(29,6){\small 2}
\put(39,6){\small 2}
\put(19,8){\small 1}
\end{picture}
\]
\[\begin{array}{|c|c|}
\cline{1-2}
5 & 6\\
\cline{1-2}
3 & 4\\
\cline{1-2}
1 & 2\\
\cline{1-2}
\end{array}
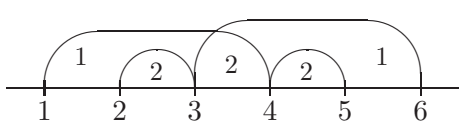
\hspace{0.5in} 
\begin{picture}(60, 10)(0, 8)
\put(0,5){\line(1,0){60}}
\put(5,4){\line(0,1){2}}
\put(15,4){\line(0,1){2}}
\put(25,4){\line(0,1){2}}
\put(35,4){\line(0,1){2}}
\put(45,4){\line(0,1){2}}
\put(55,4){\line(0,1){2}}
\put(5,2){\makebox(0,0){$1$}}
\put(15,2){\makebox(0,0){$2$}}
\put(25,2){\makebox(0,0){$3$}}
\put(35,2){\makebox(0,0){$4$}}
\put(45,2){\makebox(0,0){$5$}}
\put(55,2){\makebox(0,0){$6$}}
\put(20,5){\oval(10,10)[t]}
\put(40,5){\oval(10,10)[t]}
\put(40,5){\oval(30,18)[t]}
\put(20,5){\oval(30,15)[t]}
\put(19,6){\small 2}
\put(29,7){\small 2}
\put(39,6){\small 2}
\put(9,8){\small 1}
\put(49,8){\small 1}
\end{picture}
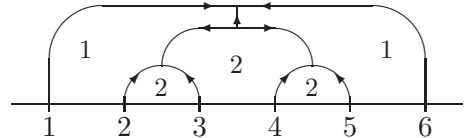
\hspace{0.5in} 
\begin{picture}(60, 15)(0, 8)
\put(0,5){\line(1,0){60}}
\put(5,4){\line(0,1){2}}
\put(15,4){\line(0,1){2}}
\put(25,4){\line(0,1){2}}
\put(35,4){\line(0,1){2}}
\put(45,4){\line(0,1){2}}
\put(55,4){\line(0,1){2}}
\put(5,2){\makebox(0,0){$1$}}
\put(15,2){\makebox(0,0){$2$}}
\put(25,2){\makebox(0,0){$3$}}
\put(35,2){\makebox(0,0){$4$}}
\put(45,2){\makebox(0,0){$5$}}
\put(55,2){\makebox(0,0){$6$}}
\put(20,5){\oval(10,10)[t]}
\put(40,5){\oval(10,10)[t]}
\put(30,10){\oval(20,10)[t]}
\put(30,5){\oval(50,26)[t]}
\put(30,15){\line(0,1){3}}
\put(25,18){\vector(1,0){2}}
\put(35,18){\vector(-1,0){2}}
\put(30,15){\vector(0,1){2}}
\put(27,15){\vector(-1,0){2}}
\put(33,15){\vector(1,0){2}}
\put(15,7){\vector(1,1){2}}
\put(25,7){\vector(-1,1){2}}
\put(35,7){\vector(1,1){2}}
\put(45,7){\vector(-1,1){2}}
\put(19,6){\small 2}
\put(29,9){\small 2}
\put(39,6){\small 2}
\put(9,11){\small 1}
\put(49,11){\small 1}
\end{picture}
\]
\[\begin{array}{|c|c|}
\cline{1-2}
4 & 6\\
\cline{1-2}
2 & 5\\
\cline{1-2}
1 & 3\\
\cline{1-2}
\end{array}
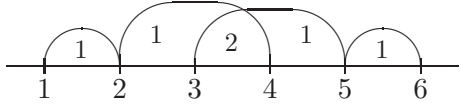
\hspace{0.5in} 
\begin{picture}(60, 20)(0, 6)
\put(0,5){\line(1,0){60}}
\put(5,4){\line(0,1){2}}
\put(15,4){\line(0,1){2}}
\put(25,4){\line(0,1){2}}
\put(35,4){\line(0,1){2}}
\put(45,4){\line(0,1){2}}
\put(55,4){\line(0,1){2}}
\put(5,2){\makebox(0,0){$1$}}
\put(15,2){\makebox(0,0){$2$}}
\put(25,2){\makebox(0,0){$3$}}
\put(35,2){\makebox(0,0){$4$}}
\put(45,2){\makebox(0,0){$5$}}
\put(55,2){\makebox(0,0){$6$}}
\put(10,5){\oval(10,10)[t]}
\put(50,5){\oval(10,10)[t]}
\put(25,5){\oval(20,17)[t]}
\put(35,5){\oval(20,15)[t]}
\put(19,8){\small 1}
\put(29,7){\small 2}
\put(39,8){\small 1}
\put(9,6){\small 1}
\put(49,6){\small 1}
\end{picture}
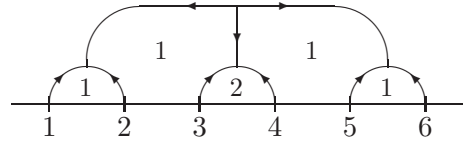
\hspace{0.5in} 
\begin{picture}(60, 20)(0, 6)
\put(0,5){\line(1,0){60}}
\put(5,4){\line(0,1){2}}
\put(15,4){\line(0,1){2}}
\put(25,4){\line(0,1){2}}
\put(35,4){\line(0,1){2}}
\put(45,4){\line(0,1){2}}
\put(55,4){\line(0,1){2}}
\put(5,2){\makebox(0,0){$1$}}
\put(15,2){\makebox(0,0){$2$}}
\put(25,2){\makebox(0,0){$3$}}
\put(35,2){\makebox(0,0){$4$}}
\put(45,2){\makebox(0,0){$5$}}
\put(55,2){\makebox(0,0){$6$}}
\put(10,5){\oval(10,10)[t]}
\put(30,5){\oval(10,10)[t]}
\put(50,5){\oval(10,10)[t]}
\put(30,10){\oval(40,16)[t]}
\put(30,10){\line(0,1){8}}
\put(5,7){\vector(1,1){2}}
\put(15,7){\vector(-1,1){2}}
\put(25,7){\vector(1,1){2}}
\put(35,7){\vector(-1,1){2}}
\put(45,7){\vector(1,1){2}}
\put(55,7){\vector(-1,1){2}}
\put(25,18){\vector(-1,0){2}}
\put(35,18){\vector(1,0){2}}
\put(30,15){\vector(0,-1){2}}
\put(19,11){\small 1}
\put(29,6){\small 2}
\put(39,11){\small 1}
\put(9,6){\small 1}
\put(49,6){\small 1}
\end{picture}
\]
Comparing the faces immediately to the left and right of a boundary vertex, the reader may also notice that depth increases if the boundary vertex is the first vertex of an $\EM$, stays the same if the vertex is the second vertex of an $\EM$, and decreases if the vertex is the third vertex of an $\EM$.  We prove this in the next section as well.  The reverse of this process constructs a three-row tableau from a web whose boundary vertices are all sources: if depth increases at a boundary vertex $i$, put $i$ on the bottom row of a tableau; if depth stays the same at $i$, put $i$ on the middle row; and if depth decreases, put $i$ on the top row of the tableau.  We show at the end of this section that this is the inverse of the map from Young tableaux to webs that we just defined.

\subsection{Analyzing path depth and circle depth in $\EM$-diagrams}
The number and relative position of faces are the same in a planar graph and its resolution.  Hence we may compare circle depth of a graph obtained from upper semicircles and path depth of its resolution.  When we do, we find that circle depth and path depth agree.  We will also see that depth is closely related to $\EM$s in the $\EM$-diagram.  The proofs in this section only use the undirected graph underlying each web.

\begin{lemma} \label{lemma: circle depth is path depth}
Let ${\mathcal C}$ be a sequence of upper semicircles such that at most two semicircles intersect at each given point.  Let $G_{\mathcal C}$ be the planar graph determined by ${\mathcal C}$ and let $G_{\mathcal C}^r$ be the resolution of $G_{\mathcal C}$.  If $x$ is a point on a face of both $G_{\mathcal C}$ and $G_{\mathcal C}^r$ then $d^c(x, {\mathcal C}) = d^p(x, G_{\mathcal C}^r)$.  
\end{lemma}

\begin{proof}
Path depth is well-defined on planar graphs, and in particular is independent of the angle or (nonzero) length of its edges.  We assume without loss of generality that the resolution producing $G_{\mathcal C}^r$ adds a very small vertical edge at each arc crossing.  

We induct on the number of semicircles in ${\mathcal C}$.  The induction hypothesis is a slightly stronger claim: any path that is a vertical line from $x$ to $f_0$, except for a very small semicircle around any resolved arc crossing, crosses the minimum number of edges possible between $x$ and $f_0$.  (The semicircle is sufficiently small if it stays inside of the local neighborhood depicted in Figure \ref{trivalizing vertices}.)  We call these paths {\em vertical paths} from $x$ to $f_0$.

When $|{\mathcal C}|$ is zero or one, the claim is trivially true.  Assume that the claim holds when $|{\mathcal C}| = n-1$.  Choose an upper semicircle $C \not \in \mathcal{C}$ satisfying the hypothesis of the lemma, and let $\mathcal{C}' = \mathcal{C} \cup \{C\}$.  

For each point $x$ on a face in $G_{{\mathcal C}'}$, we know that $d^c(x, \mathcal{C}') = d^c(x, \mathcal{C})$ if $x$ is not below $C$ and $d^c(x, \mathcal{C}') = d^c(x, \mathcal{C})+1$ if $x$ is below $C$. 

Consider any vertical path from the point $x$ to $f_0$ in $G_{\mathcal C}^r$.  If $x$ is below $C$ then each vertical path crosses $C$ exactly once, so $d^p(x,\mathcal{C}') \leq d^c(x,\mathcal{C})+1$.  At the same time, any path from $x$ to $f_0$ crosses $C$ at least once, since $C$ together with the boundary line form a closed curve.  Hence $d^p(x,\mathcal{C}') \geq d^c(x,\mathcal{C})+1$.  We conclude that if $x$ is below $C$ then $d^p(x,\mathcal{C}') = d^c(x,\mathcal{C})+1$ and each vertical path from $x$ to $f_0$ is a minimal-length path.  

Similarly, if $x$ is not below $C$, then any vertical path from $x$ to $f_0$ crosses only edges obtained from $\mathcal{C}$.  So $d^p(x,\mathcal{C}') \leq d^p(x,\mathcal{C})$.  No path from $x$ to $f_0$ in the graph associated to $\mathcal{C}'$ can cross fewer edges, else the same path can be considered in the graph associated to $\mathcal{C}$, where it contradicts the assumption on $d^p(x,\mathcal{C})$.  So $d^p(x,\mathcal{C}') = d^p(x,\mathcal{C})$.

This proves the claim.
\end{proof}

Let $(i,j,k)$ be the boundary vertices of an arbitrary $\EM$ in an $\EM$-diagram.  Colloquially, depth decreases at $i$, stays the same at $j$, and increases at $k$.   The following result makes this precise.

\begin{corollary} \label{corollary:depth map on ms}
Fix a resolved three-row $\EM$-diagram.  For each boundary vertex $i$ of a web, let $i-\epsilon$ denote a point on the face to the left of $i$ and $i+\epsilon$ denote a point on the face to the right of $i$.
\begin{itemize}
\item If $i$ is the first boundary vertex of an $\EM$ or an isolated arc, then $d(i+\epsilon) - d(i-\epsilon) = 1$.
\item If $j$ is the second boundary vertex of an $\EM$, then $d(j+\epsilon) - d(j-\epsilon) = 0$.
\item If $k$ is the last boundary vertex of an $\EM$ or an isolated arc, then $d(k+\epsilon) - d(k-\epsilon) = -1$.
\end{itemize}
\end{corollary}

\begin{proof}
Given a resolved $\EM$-diagram $M$, let $F_M$ be the set 
\[F_M = \left\{(i,j,k) \textup{ is an $\EM$ in the web} \right\} \cup \left\{(i,k) \textup{ is an isolated arc}\right\}.\]  
We use induction on the cardinality of the set $F_M$.  When $|F_M|=0$ the claim is vacuously true.  Assume that it holds when $|F_M|=n-1$. Let $M$ be an $\EM$-diagram with $|F_M|=n$.  Suppose either $(i,j,k)$ are the boundary vertices of an $\EM$ in $M$ or $(i,k)$ are the boundary vertices of an isolated arc.  Let $M'$ be the $\EM$-diagram with $|F_{M'}|=n-1$ obtained by erasing the $\EM$ with boundary vertices $(i,j,k)$, respectively the arc $(i,k)$.  The faces to the right and left of each boundary vertex $i,j,k$ merge in $M'$, so points on these faces have the same circle depth in $M'$.  Comparing to $M$, we see:
\begin{itemize}
\item the face to the left of $i$ has depth one less than the face to the right of $i$, since the face to the right of $i$ is under an arc (either $(i,j)$ or $(i,k)$) while the face to the left is unchanged from $M'$;
\item the face to the left of $j$ has the same depth as the face to the right of $j$, since the former is under the arc $(i,j)$ while the latter is under the arc $(j,k)$; and 
\item the face to the left of $k$ has depth one greater than the face to the right of $k$, since the former is under an arc (either $(j,k)$ or $(i,k)$) while the latter is unchanged from $M'$.  
\end{itemize}
If $i'$ is any vertex with $i'<i$ or $i'>k$ then the faces left and right of $i'$ are under the same arcs in $M'$ as in $M$.  If $i'$ is any vertex with $i<i'<k$ and $i' \neq j$ then the faces left and right of $i'$ are under exactly one more arc in $M$ than in $M'$.  In all cases, the claim holds.
\end{proof}

\subsection{Khovanov-Kuperberg's depth map}

We now define a map from irreducible webs for $\mathfrak{sl}_3$ {\em whose boundary vertices are all sources} to standard Young tableaux of size $3 \times n$.  The results described below were originally proven by Khovanov-Kuperberg \cite{KK99}.  We use the exposition of Petersen-Pylyavskyy-Rhoades \cite{PPR09}.  As before $i+\epsilon$ is any point on the face immediately to the right of the boundary vertex $i$, and $i-\epsilon$ is any point on the face to the left of $i$.

\begin{definition}
Given an irreducible web for $\mathfrak{sl}_3$ whose boundary vertices are sources, the {\em depth map} creates a corresponding Young tableau by inserting each boundary vertex $i$ as follows:
\[\textup{ Put } i \textup{ on the } \begin{array}{l} \textup{ top }\\ \textup{ middle } \\ \textup{ bottom } \end{array} \textup{ row of the Young tableau if } d(i+\epsilon)-d(i-\epsilon) = \begin{array}{r} -1 \\ 0 \\ 1 \end{array}\]
\end{definition}

It is not a priori clear that this map is well-defined, nor that the resulting Young tableaux have shape $3 \times n$, but in fact both statements are true.

\begin{proposition} \label{proposition:depth map well-defined}
{\bf (Khovanov-Kuperberg Lemmas 1-3 and Proposition 1, as described in Petersen-Pylyavskyy-Rhoades Theorem 2.4)} The depth map is a well-defined map from irreducible webs for $\mathfrak{sl}_3$ with all boundary vertices sources to standard Young tableaux of size $3 \times n$.
\end{proposition}

We see immediately that the depth map is the inverse of the map that sends a Young tableau of shape $(n,n,n)$ to its resolved $\EM$-diagram.

\begin{theorem}\label{theorem: bijection}
The depth map is the inverse of the map from standard Young tableaux of size $3 \times n$ to irreducible webs for $\mathfrak{sl}_3$ with boundary vertices all sources obtained by taking resolved $\EM$-diagrams.  Both maps are bijections.
\end{theorem}

\begin{proof}
By Corollary \ref{corollary:depth map on ms}, the depth map sends the first vertex of each $\EM$ to the bottom row of the Young tableau, the second vertex to the middle row, and the third vertex to the top row.  By construction of $\EM$-diagrams, this is the original Young tableau.  Kuperberg proved in  \cite[Theorem 6.1]{Kup96} that the set of irreducible webs for $\mathfrak{sl}_3$ with $3n$ boundary vertices (all sources) has the same cardinality as the set of standard Young tableaux of shape $3 \times n$.  So the claim holds.
\end{proof}

We can extend the depth map to include some irreducible webs for $\mathfrak{sl}_3$ with boundary vertices that are sinks.

\begin{definition}
Given an irreducible web for $\mathfrak{sl}_3$, the {\em extended depth map} creates a corresponding three-row Young tableau by inserting each boundary vertex $i$ as follows:
\[\textup{ Put } i \textup{ on the } \begin{array}{l} \textup{ top }\\ \textup{ middle } \\ \textup{ middle } \\ \textup{ bottom } \end{array} \textup{ row if } d(i+\epsilon)-d(i-\epsilon) = \begin{array}{r} - \\ - \\  \\  \\ \end{array}\hspace{-1em} \begin{array}{l} 1 \textup{ and $i$ is a source} \\ 1 \textup{ and $i$ is a sink} \\ 0 \\ 1 \end{array}\]
\end{definition}

The domain of the extended depth map includes irreducible webs whose boundary vertices are both sources and sinks.  The extended depth map coincides with the ordinary depth map for an irreducible web with no sinks on its boundary.  We can generalize the previous corollary as well.

\begin{proposition}
Fix $n \leq k$.  The extended depth map is well-defined on irreducible webs for $\mathfrak{sl}_3$ that are resolved $\EM$-diagrams for standard Young tableaux of shape $(n,k,k)$.  For those webs, the extended depth map is the inverse of the map that takes a three-row Young tableau to its resolved $\EM$-diagram.
\end{proposition}

The $\EM$-diagram for a tableau of shape $(n,k,k)$ has $k-n$ isolated arcs and $n$ $\EM$s, so its resolution has $3n+(k-n)$ sources and $(k-n)$ sinks on the boundary line.  When $k=n$, the boundary has no sinks and this proposition reduces to the previous corollary.

\begin{proof}
Corollary \ref{corollary:depth map on ms} together with our conventions for resolving an $\EM$-diagram show that the extended depth map is well-defined on resolved $\EM$-diagrams, and that it inverts the map from Young tableaux to resolved $\EM$-diagrams.  
\end{proof}

Unlike standard Young tableaux of shape $(n,k,k)$, arbitrary three-row standard Young tableaux have isolated boundary vertices.  Webs with isolated vertices are not irreducible (though we could extend the depth map to these webs as well).

\section{Applications}
A combinatorial spider has several natural graph-theoretic operations that correspond to essential algebraic operations on the corresponding representations: {\em rotation} of a web, {\em join} of two webs (which inserts one web into another), and {\em stitch} of a web (which connects two strands of a web).  In what follows, we show natural operations on $3 \times n$ Young tableaux that correspond to rotation and join of webs; the proofs use resolved $\EM$-diagrams, and are short and geometrically intuitive.  The operations we describe apply to all tableaux, but only correspond to operations on webs for $3 \times n$ tableaux (as we will discuss).

\subsection{Promotion and rotation}
Jeu de taquin is a classical operation on Young tableaux in which an empty box percolates to the boundary of a tableau.  In a single step on the configuration
\[\begin{array}{|c|c|}
\cline{1-1} a & \multicolumn{1}{c}{$ $} \\
\cline{1-2} & b \\
\hline
\end{array}\]
the number $a$ slides down if $a<b$ and the number $b$ slides left if $b<a$.  (Numbers outside of the tableau are considered to be $\infty$.)  Jeu-de-taquin promotion is the operation on standard tableaux obtained by
\begin{itemize}
\item erasing $1$, 
\item performing jeu-de-taquin slides until a new Young tableaux is obtained, 
\item and then adding $n$ to the newly-empty spot.  
\end{itemize}
Petersen-Pylyavskyy-Rhoades recently proved that jeu-de-taquin promotion on $3 \times n$ standard tableaux corresponds to rotation of webs \cite{PPR09}.  This was a key step in studying a cyclic sieving phenomenon, to analyze the orbits of the permutation action on tableaux obtained by promotion.  

Jeu de taquin has a natural interpretation in terms of arcs in an $\EM$-diagram.

\begin{lemma} \label{lemma:jeu de taquin slides}
Fix a standard Young tableau of arbitrary shape.  Choose $b_i$ from a row whose entries are $b_1, b_2, \ldots$ and suppose that $t_1, t_2, \ldots $ are the entries in the row above $b_i$.  After removing $b_i$ and performing jeu de taquin, the number $t_k$ slides down if and only if $t_k$ is the largest number on its row that forms an arc with a number $b_{j_0} \leq b_i$ and $k \geq i$.  
\end{lemma}

\begin{proof}
This proof involves only the subtableau consisting of the row with $b_i$ and the row above $b_i$.  For convenience, we refer to the row with $t_1, t_2, \ldots$ as the {\em top row} and the row with $b_1, b_2, \ldots$ as the {\em bottom row}, though there may be other rows in the entire Young tableau.  We sketch a schematic below.  Not all boxes are shown in the sketch; the row with $t_k$ has length {\em at least} $k$, and the row originally containing $b_i$ has length {\em at least} $k+1$.
\[\begin{array}{|c|c|c|c|c|c|c|c}
\cline{1-6} t_1 &  \cdots     & t_i & t_{i+1} & \cdots & t_k & \multicolumn{1}{c}{\cdots} &\\
\cline{1-7} b_1 &  \cdots & b_i & b_{i+1} & \cdots & b_k & b_{k+1} & \cdots \\
\cline{1-7}
\end{array}
\]

The rules of jeu de taquin imply that at most one box in each row can slide down.  Let $t_k$ be the number that slides down from the top row.  By definition $t_k$ slides down if and only if $t_j > b_{j+1}$ for $j=i,i+1,\ldots,k-1$ while $t_k < b_{k+1}$.  

We show that the inequality $t_k < b_{k+1}$ holds if and only if the arcs give a bijection between $\{t_1, t_2, \ldots, t_k\}$ and $\{b_1, b_2, \ldots, b_k\}$.  If $t_k > b_{k+1}$ then at least one of $t_1, t_2, \ldots, t_k$ forms an arc with $b_{k+1}$.  Conversely let $t_k < b_{k+1}$.  The definition of a standard tableau implies that each of $t_1, t_2, \cdots,  t_{k-1}$ is less than $b_{k+1}$, and less than every number on the bottom row and to the right of $b_{k+1}$.  By construction of $\EM$-diagrams, each of $t_1, t_2, \ldots, t_k$ is joined by an arc to a number in the set $\{b_1, b_2, \ldots, b_k\}$.  At most one of the numbers $b_1, b_2, \ldots, b_k$ is on each arc, so the arcs give a bijection as claimed.

Consider the arc $(b_{j_0}, t_{k})$.  The numbers on the bottom row under $(b_{j_0}, t_{k})$ are in bijection with the numbers on the top row under the arc $(b_{j_0}, t_{k})$ because arcs are noncrossing.  We know that $b_k < t_k < b_{k+1}$ so $\{b_{j_0}, b_{j_0+1}, \ldots, b_k\}$ are the numbers on the bottom row under $(b_{j_0}, t_{k})$.   Comparing cardinalities, we see that $\{t_{j_0}, t_{j_0+1}, \ldots, t_k\}$ are the numbers on the top row under $(b_{j_0}, t_k)$.

We conclude that $(b_{j_0},t_k)$ is an arc if and only if the arcs are a bijection from $\{b_{j_0}, b_{j_0+1}, \ldots, b_k\}$ to $\{t_{j_0}, t_{j_0+1}, \ldots, t_k\}$ that matches $b_{j_0}$ with $t_k$.  Since $t_k < b_{k+1}$ the arcs match $\{b_1, b_2, \ldots, b_k\}$ with $\{t_1, t_2, \ldots, t_k\}$, so in fact $(b_{j_0}, t_k)$ is an arc if and only if the arcs form a bijection from $\{t_1, t_2, \ldots, t_{j_0-1}\}$ to $\{b_1, b_2, \ldots, b_{j_0-1}\}$.   The arcs give a bijection between $\{b_1, b_2, \ldots, b_{j_0-1}\}$ and $\{t_1, t_2, \ldots, t_{j_0-1}\}$ if and only if $t_{j_0-1} < b_{j_0}$ by the earlier argument with $j_0-1$ replacing $k$.  In particular $j_0 \leq i$ and no number on the same row as $t_k$ and larger than $t_k$ forms an arc with any number on the bottom row that is less than $b_i$.
\end{proof}

We give a direct, short proof that promotion corresponds to rotation of webs, using resolved $\EM$-diagrams.

\begin{proposition}\label{proposition: promotion and rotation}
Jeu-de-taquin promotion on $3 \times n$ standard Young tableaux corresponds to rotation of webs for $\mathfrak{sl}_3$.
\end{proposition}

\begin{proof}
We begin by rotating an $\EM$-diagram.  Suppose that $(c,a,b)$ is an $\EM$ with first arc $(c,a)$.  Any arcs $(x_1,y_1), (x_2,y_2), \ldots, (x_k,y_k)$ that cross $(c,a)$ are second arcs.  If we rotate $c$ from the far left to the far right position on the number line, then a second arc appears to cross second arcs.  However, the two pieces of $\EM$-diagrams in Figure \ref{figure:rotating the first arc} both have the same resolution.  So the original $\EM$-diagram gives the same web as the figure with $(a,x_k,c)$ as an $\EM$, and with arcs $(b,y_1), (x_1,y_2), \ldots, (x_{k-1},y_k)$ instead of $(x_1,y_1), (x_2,y_2), \ldots, (x_k,y_k)$.  (The arcs involving $b,x_1,x_2,\ldots,x_{k-1}, y_1,y_2,\ldots,y_k$ are second arcs in both $\EM$-diagrams, so only cross first arcs.)

\begin{figure}[h]
\begin{picture}(26,50)(-13,-10)
\put(0,0){\line(0,1){40}}
\put(10,0){\line(-1,1){10}}
\put(10,15){\line(-1,0){20}}
\put(10,25){\line(-1,0){20}}
\put(5,28){$\vdots$}
\put(-5,28){$\vdots$}
\put(10,35){\line(-1,0){20}}
\put(0,10){\circle*{1}}

\put(0,39){\vector(0,-1){2}}
\put(0,31){\vector(0,-1){2}}
\put(0,21){\vector(0,-1){2}}
\put(0,14){\vector(0,-1){2}}

\put(6,15){\vector(-1,0){2}}
\put(6,25){\vector(-1,0){2}}
\put(6,35){\vector(-1,0){2}}

\put(-4,15){\vector(-1,0){2}}
\put(-4,25){\vector(-1,0){2}}
\put(-4,35){\vector(-1,0){2}}

\put(0,4){\vector(0,1){2}}
\put(6,4){\vector(-1,1){2}}

\put(-1,-3){\small $a$}
\put(10,-3){\small $b$}
\put(11,14){\small $x_1$}
\put(11,24){\small $x_2$}
\put(11,34){\small $x_k$}
\put(-14,14){\small $y_1$}
\put(-14,24){\small $y_2$}
\put(-14,34){\small $y_k$}
\put(-1,41){\small $c$}

\put(-15,-10){Piece of $\EM$-diagram}
\end{picture}
$\hspace{0.75in}$
\begin{picture}(26,50)(-13,-10)
\put(0,0){\line(0,1){40}}
\put(10,0){\line(-1,1){10}}

\put(10,19){\line(-1,0){10}}
\put(10,27){\line(-1,0){10}}
\put(5,30){$\vdots$}
\put(10,35){\line(-1,0){10}}

\put(0,15){\line(-1,0){10}}
\put(0,23){\line(-1,0){10}}
\put(-5,26){$\vdots$}
\put(0,31){\line(-1,0){10}}

\put(0,32){\vector(0,1){2}}
\put(0,24){\vector(0,1){2}}
\put(0,16){\vector(0,1){2}}

\put(0,39){\vector(0,-1){2}}
\put(0,30){\vector(0,-1){2}}
\put(0,22){\vector(0,-1){2}}
\put(0,14){\vector(0,-1){2}}

\put(6,19){\vector(-1,0){2}}
\put(6,27){\vector(-1,0){2}}
\put(6,35){\vector(-1,0){2}}

\put(-4,15){\vector(-1,0){2}}
\put(-4,23){\vector(-1,0){2}}
\put(-4,31){\vector(-1,0){2}}

\put(0,4){\vector(0,1){2}}
\put(6,4){\vector(-1,1){2}}

\put(-1,-3){\small $a$}
\put(10,-3){\small $b$}
\put(11,18){\small $x_1$}
\put(11,26){\small $x_2$}
\put(11,34){\small $x_k$}
\put(-14,14){\small $y_1$}
\put(-14,22){\small $y_2$}
\put(-14,30){\small $y_k$}
\put(-1,41){\small $c$}

\put(-10,-10){Resolution}
\end{picture}
$\hspace{0.75in}$
\begin{picture}(26,50)(-13,-10)
\put(0,0){\line(0,1){40}}
\put(10,35){\line(-1,0){10}}
\put(0,35){\circle*{1}}

\put(10,0){\line(-1,1){11}}
\put(-1,11){\line(0,1){4}}

\put(10,19){\line(-1,0){11}}
\put(10,27){\line(-1,0){11}}

\put(-1,19){\line(0,1){4}}
\put(-1,11){\line(0,1){4}}
\put(-1,27){\line(0,1){4}}

\put(-1,15){\line(-1,0){10}}
\put(-1,23){\line(-1,0){10}}
\put(-1,31){\line(-1,0){10}}

\put(-5,26){$\vdots$}
\put(5,22){$\vdots$}

\put(6,4){\vector(-1,1){2}}

\put(0,39){\vector(0,-1){2}}

\put(0,4){\vector(0,1){2}}
\put(0,14){\vector(0,1){2}}
\put(0,22){\vector(0,1){2}}
\put(0,30){\vector(0,1){2}}

\put(6,19){\vector(-1,0){2}}
\put(6,27){\vector(-1,0){2}}
\put(6,35){\vector(-1,0){2}}

\put(-4,15){\vector(-1,0){2}}
\put(-4,23){\vector(-1,0){2}}
\put(-4,31){\vector(-1,0){2}}

\put(-1,-3){\small $a$}
\put(10,-3){\small $b$}
\put(11,18){\small $x_1$}
\put(11,26){\small $x_{k-1}$}
\put(11,34){\small $x_k$}
\put(-15,14){\small $y_1$}
\put(-15,22){\small $y_2$}
\put(-15,30){\small $y_k$}
\put(-1,41){\small $c$}
\put(-15,-10){Piece of $\EM$-diagram}
\end{picture}
\caption{Two pieces of $\EM$-diagrams with the same resolution} \label{figure:rotating the first arc}
\end{figure}

The arc from $x_k$ used to be the beginning of a second arc, and is now the resolved boundary vertex on two arcs.  A resolved $\EM$-diagram should have no arcs crossing the arc from its resolved boundary vertex, though there may currently be some first arcs crossing $x_k$.  The two pieces of an $\EM$-diagram shown in Figure \ref{figure:rotating the second arc} have the same resolution.  
\begin{figure}[h]
\begin{picture}(26,50)(-13,-10)
\put(0,0){\line(0,1){40}}
\put(10,10){\line(-1,0){20}}
\put(10,20){\line(-1,0){20}}
\put(10,30){\line(-1,0){20}}
\put(5,24){$\vdots$}
\put(-5,24){$\vdots$}
\put(10,40){\line(-1,0){20}}
\put(0,40){\circle*{1}}

\put(4,10){\vector(1,0){2}}
\put(4,20){\vector(1,0){2}}
\put(4,30){\vector(1,0){2}}
\put(6,40){\vector(-1,0){2}}

\put(-6,10){\vector(1,0){2}}
\put(-6,20){\vector(1,0){2}}
\put(-6,30){\vector(1,0){2}}
\put(-6,40){\vector(1,0){2}}

\put(0,4){\vector(0,1){2}}
\put(0,14){\vector(0,1){2}}
\put(0,24){\vector(0,1){2}}
\put(0,34){\vector(0,1){2}}

\put(-1,-3){\small $x_k$}
\put(11,9){\small $v_1$}
\put(11,19){\small $v_2$}
\put(11,29){\small $v_j$}
\put(-15,9){\small $u_1$}
\put(-15,19){\small $u_2$}
\put(-15,29){\small $u_j$}
\put(11,39){\small $c$}
\put(-15,39){\small $a$}
\put(-15,-10){Piece of $\EM$-diagram}
\end{picture}
$\hspace{0.75in}$
\begin{picture}(26,50)(-13,-10)
\put(0,0){\line(0,1){40}}
\put(10,12){\line(-1,0){10}}
\put(10,22){\line(-1,0){10}}
\put(10,32){\line(-1,0){10}}
\put(0,8){\line(-1,0){10}}
\put(0,18){\line(-1,0){10}}
\put(0,28){\line(-1,0){10}}
\put(5,26){$\vdots$}
\put(-5,22){$\vdots$}
\put(10,40){\line(-1,0){20}}
\put(0,40){\circle*{1}}

\put(4,12){\vector(1,0){2}}
\put(4,22){\vector(1,0){2}}
\put(4,32){\vector(1,0){2}}
\put(6,40){\vector(-1,0){2}}

\put(-6,8){\vector(1,0){2}}
\put(-6,18){\vector(1,0){2}}
\put(-6,28){\vector(1,0){2}}
\put(-6,40){\vector(1,0){2}}

\put(0,11){\vector(0,-1){2}}
\put(0,21){\vector(0,-1){2}}
\put(0,31){\vector(0,-1){2}}

\put(0,4){\vector(0,1){2}}
\put(0,14){\vector(0,1){2}}
\put(0,24){\vector(0,1){2}}
\put(0,35){\vector(0,1){2}}

\put(-1,-3){\small $x_k$}
\put(11,11){\small $v_1$}
\put(11,21){\small $v_2$}
\put(11,31){\small $v_j$}
\put(-15,7){\small $u_1$}
\put(-15,17){\small $u_2$}
\put(-15,27){\small $u_j$}
\put(11,39){\small $c$}
\put(-15,39){\small $a$}
\put(-11,-10){Resolution}
\end{picture}
$\hspace{0.75in}$
\begin{picture}(26,50)(-13,-10)
\put(0,0){\line(0,1){40}}
\put(10,12){\line(-1,0){11}}
\put(10,22){\line(-1,0){11}}
\put(10,32){\line(-1,0){11}}
\put(0,8){\line(-1,0){10}}
\put(-1,18){\line(-1,0){10}}
\put(-1,28){\line(-1,0){10}}
\put(-1,38){\line(-1,0){10}}
\put(5,16){$\vdots$}
\put(-5,22){$\vdots$}
\put(10,40){\line(-1,0){10}}
\put(0,8){\circle*{1}}

\put(-1,18){\line(0,-1){6}}
\put(-1,28){\line(0,-1){6}}
\put(-1,38){\line(0,-1){6}}

\put(4,12){\vector(1,0){2}}
\put(4,22){\vector(1,0){2}}
\put(4,32){\vector(1,0){2}}
\put(6,40){\vector(-1,0){2}}

\put(-6,8){\vector(1,0){2}}
\put(-6,18){\vector(1,0){2}}
\put(-6,28){\vector(1,0){2}}
\put(-6,38){\vector(1,0){2}}

\put(0,11){\vector(0,-1){2}}
\put(0,18){\vector(0,-1){2}}
\put(0,28){\vector(0,-1){2}}

\put(0,3){\vector(0,1){2}}
\put(0,37){\vector(0,-1){2}}

\put(-1,-3){\small $x_k$}
\put(11,11){\small $v_1$}
\put(11,21){\small $v_{j-1}$}
\put(11,31){\small $v_j$}
\put(-15,7){\small $u_1$}
\put(-15,17){\small $u_2$}
\put(-15,27){\small $u_j$}
\put(11,39){\small $c$}
\put(-15,37){\small $a$}
\put(-15,-10){Piece of $\EM$-diagram}
\end{picture}
\caption{Two pieces of $\EM$-diagrams with the same resolution} \label{figure:rotating the second arc}
\end{figure}
So the original $\EM$-diagram gives the same web as the figure with $(u_1,x_k,c)$ as an $\EM$, and with arcs $(a,v_j), (u_j,v_{j-1}), \ldots, (u_2,v_1)$ instead of $(u_1,v_1), (u_2,v_2), \ldots, (u_j,v_j)$.  The arcs involving $a,u_1,u_2,\ldots,u_j,v_1,v_2,\ldots,v_j$ are first arcs in both $\EM$-diagrams, so only cross second arcs.  

Hence rotating the original $\EM$-diagram so that $c$ goes from first to last position on the boundary gives the resolved $\EM$-diagram with $(u_1,x_k,c)$ as an $\EM$, and with arcs $(a,v_j), (u_j,v_{j-1}), \ldots, (u_2,v_1)$ and $(b,y_1), (x_1,y_2), \ldots, (x_{k-1},y_k)$ that otherwise agrees with the original $\EM$-diagram.

We now confirm that the tableau corresponding to this rotated $\EM$-diagram is the promotion of the original Young tableau.  It suffices to determine which numbers are on which row of the Young tableau after promotion.  After promoting the original Young tableau, the number $c=1$ is removed.  The largest number on the middle row with an arc to a number at most $1$ on the bottom row is $a$, by definition.  Lemma \ref{lemma:jeu de taquin slides} says that jeu de taquin slides $a$ down to the bottom row.  Figure \ref{figure:rotating the first arc} shows that $x_k$ is defined to be the largest number on the top row that is joined by an arc to a number $y_k$ on the middle row with $y_k \leq a$.  In other words $x_k$ is the number that slides from the top to middle row after jeu de taquin, also by Lemma \ref{lemma:jeu de taquin slides}.  The original Young tableau has three rows so these two vertical slides determine each row of the promoted Young tableau, which therefore corresponds to the $\EM$-diagram in Figure \ref{figure:rotating the second arc}.
\end{proof}

Rotation does not correspond to promotion of {\em arbitrary} three-row tableaux.  For instance, promotion gives
\[\begin{array}{|c|c|}
\cline{1-1}
3 & \multicolumn{1}{c}{} \\
\cline{1-2}
2 & 5\\
\cline{1-2}
1 & 4\\
\cline{1-2}
\end{array}  \rightarrow
\begin{array}{|c|c|}
\cline{1-1}
5 & \multicolumn{1}{c}{} \\
\cline{1-2}
2 & 4\\
\cline{1-2}
1 & 3\\
\cline{1-2}
\end{array}
\]
while rotating the web with arcs $(1,2,3)$ and $(4,5)$ gives the web with arcs $(1,2,5)$ and $(3,4)$.  The reader may notice an provocative connection between the rotated web and the promoted tableau in this case.  Perhaps our construction could be modified to extend to the general three-row case.

\subsection{Insertion and joins}
The join of two webs is obtained by cutting the boundary line of one web and slipping the second web into the gap.  We define the join of two $\EM$-diagrams analogously.  The join operation commutes with resolving $\EM$-diagrams, in the sense that the join of two resolved $\EM$-diagrams is the resolution of the join of the $\EM$-diagrams.  For instance, the join after $1$ of the $\EM$-diagram with arc $(1,2)$ into the $\EM$-diagram with $\EM$ $(1,2,3)$ inserts $(1,2)$ after the first boundary vertex of $(1,2,3)$.  This produces the following $\EM$-diagram and (undirected) web, regardless of whether the join is taken before or after resolving $\EM$-diagrams:
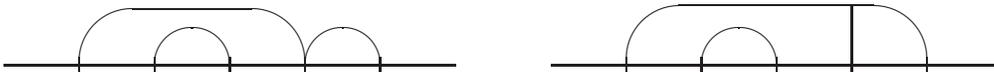
\[\begin{picture}(60,15)(0,-5)
\put(0,0){\line(1,0){60}}
\multiput(10,-1)(10,0){5}{\line(0,1){2}}
\put(25,0){\oval(10,10)[t]}
\put(25,0){\oval(30,15)[t]}
\put(45,0){\oval(10,10)[t]}
\end{picture}
\hspace{0.5in}
\begin{picture}(60,15)(0,-5)
\put(0,0){\line(1,0){60}}
\multiput(10,-1)(10,0){5}{\line(0,1){2}}
\put(25,0){\oval(10,10)[t]}
\put(30,0){\oval(40,16)[t]}
\put(40,0){\line(0,1){8}}
\end{picture}
\]
We show that join is equivalent to a natural operation on standard Young tableaux, which we call a {\em shuffle} of tableaux.

\begin{definition}
Let $T$ and $T'$ be arbitrary Young tableaux with $N$ and $N'$ boxes, respectively.  Let $i \leq N$.  The shuffle of $T'$ into $T$ at $i$ is a tableau denoted $T' \stackrel{i}{\mapsto} T$ and defined by:
\begin{itemize}
\item For $j = 1,2,\ldots, i$, put $j$ in the same row of $T' \stackrel{i}{\mapsto} T$ as in $T$.
\item For $j = 1,2,\ldots, N'$, put $i+j$ in the same row of $T' \stackrel{i}{\mapsto} T$ as $j$ is in $T'$.
\item For $j = i+1,i+2,\ldots, N$, put $N'+j$ in the same row of $T' \stackrel{i}{\mapsto} T$ as $j$ is in $T$.
\end{itemize}
\end{definition}

If $T$ and $T'$ are standard, then the shuffle is a standard tableau by construction.  We give an example; inserted numbers are in boldface and the boundary of $T'$ is highlighted.  A shuffle splits $T$ into two pieces which fit together perfectly, similar to splitting a deck of cards.
\begin{figure}[h]
$T = \begin{array}{|c|c|}
\cline{1-2}
5 & 6\\
\cline{1-2}
3 & 4\\
\cline{1-2}
1 & 2\\
\cline{1-2}
\end{array} \hspace{0.5in} $
$T' = \begin{array}{|c|c|}
\cline{1-2}
4 & 6\\
\cline{1-2}
2 & 5\\
\cline{1-2}
1 & 3\\
\cline{1-2}
\end{array} \hspace{0.5in} $
$T' \stackrel{3}{\mapsto} T = 
 \mbox{
\begin{picture}(20,15)(0,6)
\multiput(0,0)(5,0){5}{\line(0,1){15}} 
\multiput(0,0)(0,5){4}{\line(1,0){20}} 
\put(1.5,1.5){1}
\put(6.5,1.5){2}
\put(11.5,1.5){\bf 4}
\put(16.5,1.5){\bf 6}
\put(1.5,6.5){3}
\put(6.5,6.5){\bf 5}
\put(11.5,6.5){\bf 8}
\put(15.5,6.5){10}
\put(1.5,11.5){\bf 7}
\put(6.5,11.5){\bf 9}
\put(10.5,11.5){11}
\put(15.5,11.5){12}

\linethickness{2pt}
\put(10,0){\line(1,0){10}}
\multiput(20,0)(-5,5){3}{\line(0,1){5}}
\multiput(10,0)(-5,5){3}{\line(0,1){5}}
\multiput(10,0)(-5,5){3}{\line(1,0){5}}
\multiput(15,5)(-5,5){3}{\line(1,0){5}}
\put(0,15){\line(1,0){10}}
\end{picture}
}
$
\caption{The shuffle of one tableau into another}
\end{figure}

We now prove that shuffle of tableaux corresponds to join of $\EM$-diagrams.  We give two versions of the claim: the first restricts the shape of the Young tableaux but uses arbitrary shuffles; the second restricts the kind of shuffles but uses arbitrary shapes.  For three-row Young tableaux, the $\EM$-diagrams resolve to webs for $\mathfrak{sl}_3$.

\begin{proposition} \label{proposition: shuffle and join}
Let $T'$ be a Young tableau with $N'$ boxes and $T$ be a Young tableau with $N$ boxes.  
\begin{enumerate}
\item The shuffle $T' \stackrel{N}{\mapsto} T$ corresponds to the join after $N$ of the resolved $\EM$-diagram for $T'$ of the resolved $\EM$-diagram for $T$. 
\item Suppose that $T'$ has at least as many rows as $T$, and that each row of $T'$ has the same length (i.e. $T'$ is a rectangle).  The resolved $\EM$-diagram of the shuffle $T' \stackrel{i}{\mapsto} T$ is the join after $i$ of the resolved $\EM$-diagram for $T'$ into the resolved $\EM$-diagram for $T$.
\end{enumerate}
\end{proposition}

\begin{proof}
If $j=1, \ldots, i$ then $j$ is joined to $k$ by an arc in $T' \stackrel{i}{\mapsto} T$ if and only if $j$ is joined to $k$ by an arc in $T$ by construction of $\EM$-diagrams.

We next show that if $j = 1, \ldots, N'$ then $i+j$ is joined to $i+k$ by an arc in $T' \stackrel{i}{\mapsto} T$ if and only if $j$ is joined to $k$ by an arc in $T'$.  The number $i+k$ is on the row below $i+j$ in $T' \stackrel{i}{\mapsto} T$ if and only if $k$ is on the row below $j$ in $T'$.  Moreover $i+k < i+j$ if and only if $k < j$.  Assume that if $j \leq j'$ then $i+j$ is joined to $i+k$ by an arc in $T' \stackrel{i}{\mapsto} T$ if and only if $j$ is joined to $k$ by an arc in $T'$.  (This is true for $j=1$ since in that case $j$ is on the bottom row of $T'$ and $i+j$ is on the bottom row of $T' \stackrel{i}{\mapsto} T$, so neither has an arc to a lower row.)  Let $j=j'+1$ and say $i+j$ is on the $r+1^{th}$ row.  Suppose that the arcs from $1, 2, \ldots, i+j-1$ have been placed in $T' \stackrel{i}{\mapsto} T$ according to the rules of $\EM$-diagrams.  Then the set
\[\{k': k'<i+j, k' \textup{ is on the $r^{th}$ row in } T' \stackrel{i}{\mapsto} T, k' \textup{ is not yet part of an $r^{th}$ arc in } T' \stackrel{i}{\mapsto} T\} \]
contains the set
\[i+\{k: k<j, k \textup{ is on the $r^{th}$ row in } T' , k \textup{ is  not yet part of an $r^{th}$ arc in } T' \} \]
as well as perhaps some numbers that are at most $i$.  Thus the maximum of each set is the same, so $(j,k)$ is an arc in $T'$ if and only if $(i+j,i+k)$ is an arc in $T' \stackrel{i}{\mapsto} T$.  By induction, this holds for all $j=1,2,\ldots,N'$.

When $i=N$ this proves Part (1) of the claim.

Otherwise, assume that $T'$ has at least as many rows as $T$ and that $T'$ is rectangular.  Then each number $i+1, i+2, \ldots, i+N'$ is either on the top row of $T'$ and hence of $T' \stackrel{i}{\mapsto} T$ or is one vertex of an arc to the next higher row (because $T'$ is rectangular).  This means that for each $j=i+N'+1, i+N'+2,\ldots,N+N'$ with $j$ on the $r+1^{th}$ row, we may inductively show that the set
\[\{k: k<j, k \textup{ is on the $r^{th}$ row}, k \textup{ is not yet part of an $r^{th}$ arc in } T' \stackrel{i}{\mapsto} T\}\]
equals the set
\[\{k:  k \textup{ is on the $r^{th}$ row}, k \textup{ is not yet part of an $r^{th}$ arc in } T' \stackrel{i}{\mapsto} T, k<i \textup{ or } k>i+N'\}.\]
The maximum of this set always corresponds to an entry from $T$, so the maximum equals
\[\left\{\begin{array}{rl}
\max\{k: k<j-N', k \textup{ is on the $r^{th}$ row}, k \textup{ is not in an $r^{th}$ arc in } T \} & \textup{   if } k \leq i \textup{ and} \\
N'+ \max\{k: k<j-N', k \textup{ is on the $r^{th}$ row}, k \textup{ is not in an $r^{th}$ arc in } T \} & \textup{   if } k > i+N'.
\end{array} \right.
\]
In other words, the arcs in $T' \stackrel{i}{\mapsto} T$ either involve only vertices from $T'$ in the same relative positions as the arcs in $T'$, or vertices from $T$ in the same relative positions as the arcs in $T$.  So the $\EM$-diagram corresponding to $T' \stackrel{i}{\mapsto} T$ is the join of the resolved $\EM$-diagrams for $T'$ and $T$.
\end{proof}

Part (2) does not hold for arbitrary $T'$.  For instance, the shuffle
\[\begin{array}{|c|} \cline{1-1} 2 \\ \cline{1-1} 1 \\ \hline \end{array}  \stackrel{2}{\mapsto} \begin{array}{|c|} \cline{1-1} 3  \\ \cline{1-1} 2 \\ \cline{1-1} 1 \\ \hline \end{array} = 
\begin{array}{|c|c|}
\cline{1-1}
5 & \multicolumn{1}{c}{} \\
\cline{1-2}
2 & 4\\
\cline{1-2}
1 & 3\\
\cline{1-2}
\end{array} 
\]
corresponds to the $\EM$-diagram with arc $(1,2)$ and $\EM$ $(3,4,5)$.  By contrast, joining the $\EM$-diagrams $(1,2)$ to $(1,2,3)$ at $2$ gives the $\EM$-diagram with $\EM$ $(1,2,5)$ and arc $(3,4)$.

Together with rotation, Part (1) can be used to construct an arbitrary join of resolved $\EM$-diagrams.  However this gives a weaker claim than Part (2) of Proposition \ref{proposition: shuffle and join}: rotation of webs corresponds to jeu-de-taquin promotion only for tableaux of shape $(n,n,n)$, and the shuffle $T' \stackrel{N}{\mapsto} T$  is a rectangular tableau only if both $T'$ and $T$ are rectangular of the same height.

\section{Acknowledgements}
The author thanks Annie Meyers for significant contributions in early stages of this work, to Dave Anderson for useful suggestions, and to Alex Yong for suggesting the word ``shuffle".

\end{document}